\documentclass[12pt]{amsart}
\usepackage{amsmath}
\usepackage{amsthm,amsxtra,amssymb,amsfonts,latexsym, mathrsfs, dsfont, bbm, bm, verbatim, mathtools}

\usepackage[initials,alphabetic]{amsrefs}
\usepackage[usenames]{color}
\usepackage[all]{xy}
\usepackage{hyperref}
\renewcommand{\vec}[1]{\underline{#1}}

\allowdisplaybreaks[1]

\newtheorem{lemma}{Lemma}[section]
\newtheorem{prop}[lemma]{Proposition}
\newtheorem{theorem}[lemma]{Theorem}
\newtheorem{cor}[lemma]{Corollary}

\theoremstyle{remark}
\newtheorem{rem}[lemma]{Remark}

\theoremstyle{definition}

\newtheorem{exam}[lemma]{Example}
\numberwithin{equation}{section}

\newcommand{\lge}{\langle}
\newcommand{\rge}{\rangle}

\newcommand{\ax}{\mathcal{A}}

\newcommand{\fx}{\mathcal{F}}

\newcommand{\mx}{\mathcal{M}}
\newcommand{\nx}{\mathcal{N}}

\newcommand{\ux}{\mathcal{U}}

\newcommand{\nz}{\mathbb{N}}
\newcommand{\rz}{\mathbb{R}}
\newcommand{\cz}{\mathbb{C}}
\newcommand{\ez}{\mathbb{E}}
\newcommand{\fz}{\mathbb{F}}

\newcommand{\zz}{\mathbb{Z}}
\newcommand{\pz}{\mathbb{P}}

\newcommand{\Ga}{\Gamma}
\newcommand{\Om}{\Omega}
\newcommand{\De}{\Delta}

\newcommand{\al}{\alpha}
\newcommand{\de}{\delta}
\newcommand{\si}{\sigma}
\newcommand{\ga}{\gamma}
\newcommand{\la}{\lambda}
\newcommand{\bt}{\beta}

\newcommand{\om}{\omega}

\newcommand{\eps}{\varepsilon}
\newcommand{\8}{\infty}
\newcommand{\td}{\widetilde}
\mathchardef\dash="2D

\newcommand{\Var}{\operatorname{Var}}

\newcommand{\Span}{\operatorname{span}}
\newcommand{\Dom}{\operatorname{Dom}}
\newcommand{\tr}{\operatorname{tr}}
\newcommand{\id}{\operatorname{id}}

\begin{document}
\title[Mixed $q$-Gaussian algebras]{Ultraproduct methods for mixed $q$-Gaussian algebras}
\author{Marius Junge}
\address{Department of Mathematics, University of Illinois, Urbana, IL 61801}
\email{mjunge@illinois.edu}
\author{Qiang Zeng}
\address{Mathematics Department, Northwestern University,
2033 Sheridan Road Evanston, IL 60208}
\email{qzeng.math@gmail.com}
\date{\today}

\subjclass[2010]{46L36, 46L53, 46N50, 81S05}
\keywords{$q$-Gaussian algebras, Wick product, hypercontractivity, Riesz transform, Poincar\'e inequality, approximation property, strong solidity}
\maketitle

\begin{abstract}
We provide a unified ultraproduct approach for constructing Wick words in mixed $q$-Gaussian algebras, which are generated by $s_j=a_j+a_j^*$, $j=1,\cdots,N$, where $a_ia^*_j - q_{ij}a^*_ja_i =\de_{ij}$. Here we also allow equality in $-1\le q_{ij}=q_{ji}\le 1$. Using the ultraproduct method, we construct an approximate co-multiplication of the mixed $q$-Gaussian algebras. Based on this we prove that these algebras are weakly amenable and strongly solid in the sense of Ozawa and Popa. We also encode Speicher's central limit theorem in the unified ultraproduct method, and show that the Ornstein--Uhlenbeck semigroup is hypercontractive, the Riesz transform associated to the number operator is bounded, and the number operator satisfies the $L_p$ Poincar\'e inequalities with constants $C\sqrt{p}$. %The results in this paper can be regarded as generalizations of previous results due to Speicher, Biane, Lust-Piquard, Avsec, et al.
\end{abstract}

\section{Introduction}

Group measure space constructions go back to the original work of Murray and von Neumann \cite{Mv36}. In the last decades Popa and his collaborators have solved many open problems about fundamental groups and uniqueness of Cartan subalgebras; see e.g. \cites{OP10, OP10b,PV10,HS11, PV14} and the references therein for more information. In parallel von Neumann algebras generated by $q$-commutation relations (motivated by physics and number theory) were introduced by Bo\.zejko and Speicher \cite{BS91}, and further investigated by Bo\.zejko--K\"ummerer--Speicher \cite{BKS}, Shlyakhtenko \cite{Sh04}, Nou \cite{Nou}, {\'S}niady \cite{Sn04}, Ricard \cite{Ri05}, Kennedy--Nica \cite{KN11}, and Avsec \cite{Av}, among others. More recently, Dabrowski \cite{Da14}, Guionnet and Shlyakhtenko \cite{GS14} have shown that for small $q$ the $q$-Gaussian algebras are isomorphic to free group factors. All these results on factoriality, embeddability in $R^{\om}$, and approximation properties face a similar problem: How to derive properties of von Neumann algebras from combinatorial structures given by the original $q$-commutation relations.

In this paper we study generalized $q$ commutation relations: Given a symmetric matrix $Q=(q_{ij})_{i,j=1}^N$, $q_{ij}\in[-1,1]$, Speicher \cite{Sp93} considered variables satisfying
\begin{equation}\label{qcom}
  a_ia^*_j - q_{ij}a^*_ja_i =\de_{ij}.
\end{equation}
The mixed $q$-Gaussian algebra $\Ga_Q$ is generated by the self-adjoint variables $s_j=a_j^*+a_j$ and admits a normal faithful tracial state (see Section 3 for more details). Bo\.zejko and Speicher took a systematic way in \cite{BS94} to construct the Fock space representation of the so-called braid relations, which is more general than \eqref{qcom}. Then various properties were studied in e.g. \cites{Nou, Kr00,Kr05}. As for \eqref{qcom}, in \cite{LP99} Lust-Piquard showed the $L_p$ boundedness of the Riesz transforms associated to the number operator of the system when $q_{ii}<1$. Other generalized Gaussian systems related to our investigation have also been studied; see e.g. \cite{GM02,Gu03}.

It is very tempting to believe that mixed $q$-Gaussian algebras behave in any respect the same way as the $q$-Gaussian algebras with constant $q$. Indeed, the $L_2$ space of such an algebra admits a decomposition
 \[ L_2(\Gamma_Q) = \bigoplus_{k=0}^\8 H^k_Q
 \]
into finite dimension subspaces $H^k_Q$ of dimension $N^k$, which are eigenspaces of the number operator. For fixed $q_{ij}=q$ the number operator can be defined in a functorial way following Voiculescu's lead \cite{VDN} for $q=0$. Indeed, for every real Hilbert $H$ one finds the $q$-Gaussian von Neumann algebra $\Gamma_q(H)$ and a group homomorphism
$\al:O(H)\to {\rm Aut}(\Gamma_q(H))$ such that for $o\in O(H)$ and $h\in H$
 \[ \al(o)(s^q(h))= s^q(o(h)).
 \]
Here $s^q(e_j)=s_j$ and $(e_j)$ is an orthonormal basis for the $N$ dimensional Hilbert space $H$. Then
 \[ T_t = E \al(o_{t})\pi ,
 \]
where $\pi:\Gamma_q(H)\to \Gamma_q(H\oplus H)$ is  the natural embedding with conditional expectation $E$ and
$$
o_t=
\Big( \begin{array}{cc}
e^{-t}\id& -\sqrt{1-e^{-2t}}\id\\
\sqrt{1-e^{-2t}} \id & e^{-t} \id
\end{array}\Big),
$$

For nonconstant $Q=(q_{ij})$ we can no longer refer to functoriality directly. One of the first results in this paper is to provide a unified approach to the Ornstein--Uhlenbeck semigroup for $|q_{ij}|\le 1$ including the classical cases $q=1$ for bosons and $q=-1$ for fermions. The fact that the dimension of the eigenspace $H^k_Q$ is not more than $N^k$ uniformly for all $Q$ is based on thorough analysis of different forms of Wick words and probabilistic  estimates (see Section 3).

Another new feature of these generalized relations comes from studying the operators
  \[ \Phi(x) = E(s_{N+1}xs_{N+1}), \]
where $x$ is generated by $s_1,\cdots,s_N$. For constant $q_{ij}=q$ we find
$\Phi(x)=q^{l(x)}x$ can be easily computed in terms of the length function $l(x)=k$ if $x\in H_Q^{k}$. The formula for general $Q$ is vastly more complicated. However, such expressions are crucial building blocks in proving strong solidity.

\begin{comment}
  In this paper we are concerned with a general Gaussian system. Let $Q=(q_{ij})_{i,j}$ be an $N\times N$ symmetric matrix with $q_{ij}\in[-1,1]$. Speicher \cite{Sp93} introduced the following commutation relation
\begin{equation}\label{qcom}
  a_ia^*_j - q_{ij}a^*_ja_i =\de_{ij}
\end{equation}
under the name of generalized statistics of macroscopic fields in 1993; see \cite{Sp93} and the references therein for the background in physics. He also showed the existence of a Fock space representation of relations \eqref{qcom} in the same paper. In this paper, we shall call \eqref{qcom} the mixed $q$-commutation relations, call the von Neumann algebra generated by these relations the mixed $q$-Gaussian algebra $\Ga_Q$, and call $Q$ the structure matrix of $\Ga_Q$. Our goal in this paper is to study various analytic and operator algebraic properties of $\Ga_Q$.
Note that \eqref{qcom} is a generalization of the well-known $q$-commutation relation, which has been studied extensively; see e.g. \cites{BS91,BKS,Sh04, Ri05, Av} and the references therein for more details, especially for various properties of the von Neumann algebra generated by the $q$-commutation relations.
\end{comment}
Let us recall some notions in operator algebras. We always assume the von Neumann algebras to be finite in this paper. Recall that a von Neumann algebra $\mx$ has the weak* completely bounded approximation property (w*CBAP) if there exists a net of normal, completely bounded, finite rank maps $\phi_\al: \mx\to\mx$ such that $\|\phi_\al\|_{cb}\le C$ for all $\al$ and $\phi_\al\to \id$ in the point weak* topology. Here $\|\cdot\|_{cb}$ denotes the completely bounded norm. The infimum of such constant $C$ is called the Cowling--Haagerup constant and is denoted by $\Lambda_{cb}(\mx)$. Cowling and Haagerup \cite{CH89} showed that a discrete group $G$ is weakly amenable if and only if its group von Neumann algebra $LG$ has w*CBAP. Thus, a von Neumann algebra with w*CBAP is also said to be weakly amenable. If $\Lambda_{cb}(\mx)=1$, $\mx$ is said to have the weak* completely contractive approximation property (w*CCAP). See e.g. \cite{BO08} for more details of the approximation properties. Following Ozawa and Popa's work \cite{OP10}, a von Neumann algebra $\mx$ is called strongly solid if the normalizer $\nx_\mx(P):=\{u\in \ux(\mx): uPu^* = P\}$ of any diffuse amenable subalgebra $P\subset \mx$ generates an amenable von Neumann algebra. Here $\ux(\mx)$ is the set of unitary operators in $\mx$. %Strong solidity has several important consequences in the theory of operator algebras; see e.g. \cites{OP10,HS11,Av} and the references therein for more information in this direction.
\begin{theorem}
  $\Ga_Q$ has w*CCAP and is strongly solid provided $\max_{1\le i,j\le N} |q_{ij}|<1$.
\end{theorem}
These properties extend similar results due to Avsec \cite{Av} for $q$-Gaussian von Neumann algebras. The w*CCAP for $\Ga_Q$ is proved using a transference method based on Avsec's w*CCAP result for the $q$-Gaussian algebras. Then we show a weak containment result of certain bimodules. These results, together with a modification of Popa's s-malleable deformation estimate, leads to strong solidity using a, by now,  standard argument. The method used here follows that of \cites{HS11,Av}. However, the techniques are more difficult than the case of $q$-Gaussian algebras. We have to use some non-trivial tricks to achieve certain results similar to those in \cite{Av}.

Ultraproduct method plays an essential role in many aspects of this paper. It is well known by far that CCAP is a stepping stone for proving strong solidity. The transference method mentioned above relies entirely on an embedding of $\Ga_Q$ into an ultraproduct of von Neumann algebras which preserves the Wick words. This allows us to transfer the CCAP result of the constant $q$ case of Avsec to the current mixed $q$ case. The argument can be illustrated using the following commutative diagram:
 \[
  \xymatrix{
  \Ga_Q\ar@{^{(}->}^<<<<{\pi_\ux}[r] \ar[d]^{\psi_\al} & \prod_{m,\ux} \Ga_{q}(\ell_2^m)\bar\otimes \Ga_{\td Q\otimes \mathds{1}_m}\ar[d]^{\varphi_\al(A)\otimes \id}\\
  \Ga_Q\ar@{^{(}->}[r]^<<<<{\pi_\ux}  & \prod_{m,\ux} \Ga_{q}(\ell_2^m)\bar\otimes \Ga_{\td Q\otimes \mathds{1}_m}.
  }
  \]
The notation will be explained in the proof of Theorem \ref{ccap}. The map $\pi_\ux$ can be understood as an approximate co-multiplication. Without the help of the ultraproduct method above, we will have to extend directly the argument for the constant $q$ case to the mixed $q$ case, which may be very hard due to the involved combinatorial structure.

We also prove some analytic properties for $\Ga_Q$ following the unified ultraproduct approach. The cornerstone is a Wick word decomposition result, whose proof involves some complicated combinatorial and probabilistic arguments. In this context, the ultraproduct construction provides a natural framework to encode Speicher's central limit theorem; see \cite{Sp92,Sp93,Ju06}. Furthermore, the Wick words are identified as some special sequences in the ultraproduct of spin matrix models. Once we have the Wick word decomposition, it follows immediately that the Ornstein--Uhlenbeck semigroup $(T_t)_{t\ge0}$ associated to $\Ga_Q$ is hypercontractive: For $1\le p,r<\8$,
  \[
  \|T_t\|_{L_p\to L_r}=1 \quad \mbox{if and only if}\quad e^{-2t}\le \frac{p-1}{r-1}.
  \]
Here $L_p=L_p(\Ga_Q,\tau_Q)$ is the noncommutative $L_p$ space associated to the canonical tracial state $\tau_Q$ on $\Ga_Q$. This result is a vast generalization of the work of Biane \cite{Bi} and Junge et al. \cite{JPP}.
Indeed, we obtain hypercontractivity results for  free products of $q$-Gaussian algebras, in particular, free products of Clifford algebras.
More exotic choices may be obtained for general $q_{ij}$.
We also recover and extend Lust-Piquard's result on the boundedness of Riesz transforms \cite{LP99}. Let $A$ be the number operator of $\Ga_Q$, which is also the generator of $T_t$. Define the gradient form (Meyer's ``carr\'e du champs'') associated to $A$ as
\[
\Ga(f,g) = \frac12[A(f^*) g + f^* Ag- A(f^*g)]
\]
for $f,g$ in the domain of $A$. We show that a) for $p\ge 2$,
\[
  c_p^{-1}\|A^{1/2} f\|_p\le \max\{\|\Ga(f,f)^{1/2}\|_p, \|\Ga(f^*,f^*)^{1/2}\|_p\}\le K_p\|A^{1/2}f\|_p
  \]
with $c_p=O(p^{2})$ and $K_p=O(p^{3/2})$; b) for $1< p< 2$,
\[
K_{p'}^{-1} \|A^{1/2} f\|_p\le \inf_{\substack{\de(f)=g+h\\ g\in G_p^c,h\in G_p^r}} \{\|E(g^*g)^{1/2}\|_p +\|E(hh^*)^{1/2}\|_p\}\le C_p\|A^{1/2} f\|,
\]
with $K_{p'}=O(1/(p-1)^{3/2})$ and $C_p=O(1/(p-1)^2)$, where $\de$ is a derivation related to the Riesz transforms, $G_p^c$ and $G_p^r$ are two Gaussian spaces, all to be defined below. Moreover, we obtain the $L_p$ Poincar\'e inequalities: For $p\ge 2$
\[
  \|f-\tau_Q(f)\|_p\le C\sqrt{p}\max\{\|\Ga(f,f)^{1/2}\|_p, \Ga(f^*,f^*)^{1/2}\|_p\}.
\]
This is an extension of similar results for the Walsh and Fermionic system in \cite{ELP}. It is known that the constant $C\sqrt{p}$ in such inequalities is crucial for proving concentration and transportation inequalities; see e.g. \cite{Ze13}.

The paper is organized as follows. Some preliminaries and notation are presented in Section 2. We construct the mixed $q$-Gaussian algebras and the Ornstein--Uhlenbeck semigroup in Section 3, where the Wick word decomposition result is also proved with a lengthy argument. The analytic properties are proved in Section 4, and the strong solidity is proved in Section 5.

\section{Preliminaries and notation}
\subsection{Notation}\label{s:nota}
We now introduce some notation used throughout this paper. We write $[N]=\{1,2,\cdots, N\}$ for $N\in \nz$. The set of non-negative integers is denoted by $\zz_+$. For $n\in \nz$, we denote by $M_n$ the algebra of $n\times n$ matrices. We will use some notation to analyze combinatorial structures following \cite{Sp92,JPP}. Denote by $P(d)$ the set of all partitions of $[d]=\{1,\cdots,d\}$. For $\si,\pi\in P(d)$, we write $\si\le \pi$ or $\pi\ge \si$ if $\si$ is a refinement of $\pi$. We denote the integer valued vectors by $\vec{i},\vec{j}$, etc. Given $\vec{i}=(i_1,\cdots,i_d) \in [N]^d$, we associate a partition $\si(\vec{i})$ to $\vec{i}$ by requiring $k,l\in [d]$ belonging to the same block of $\si(\vec{i})$ if and only if $i_k=i_l$.

We denote by $|S|$ or $\# S$ the cardinality of a (finite) set $S$. If $d$ is an even integer, we define $P_2(d)$ to be the set of pair partitions of $[d]$, i.e., $P_2(d)$ consists of $\pi=\{V_1,\cdots,V_{d/2}\}$ such that $|V_k|=2$ for every block $V_k$. Write $V_k=\{e_k,z_k\}$ with $e_k<z_k$ and $e_1<e_2<\cdots <e_{d/2}$. Given $\pi\in P_2(d)$, the set of crossings of $\pi$ is denoted by
\begin{equation}\label{cross}
  I(\pi)=\{\{k,l\}| 1\le k,l\le d/2, e_k<e_l<z_k<z_l\}.
\end{equation}
For $d\in \nz$, we denote by $P_1(d)$ the one element set of singleton partition of $[d]$, i.e., $P_1(d)=\{\si_0\}$ and $\si_0=\{\{1\},\{2\},\cdots,\{d\}\}$. Let $P_{1,2}(d)$ denote the set of partitions consisting of only singletons and pair blocks, and $P_r(d)=P(d)\setminus P_{1,2}(d)$. Let $\si\in P_{1,2}(d)$ be given by
\[
\si=\{V_1,\cdots, V_{s+u}\},
\]
where $V_j$'s are singletons ($V_j=\{e_j=z_j\}$) or pair blocks ($V_j=\{e_j, z_j\}$). Assume there are $s$ singleton blocks and $u$ pair blocks in $\si$. Let $\si_p$ be a subpartition consisting of the $u$ pair blocks of $\si\in P_{1,2}(d)$. Denote by $I_p(\si):=I(\si_p)$ the set of pair crossings of $\si$ given in \eqref{cross} and define
\[
I_{sp}(\si)=\{\{r,t\}: e_r< e_t=z_t< z_r\}
\]
to be the set of crossings between pairs and singletons.

Given a discrete group $G$, the left regular representation is denoted by $\la: G\to \ell_2(G)$, $\la(g)\de_h=\de_{gh}$, for $g,h\in G$ and $(\de_g)_{g\in G}$ is a canonical basis of $\ell_2(G)$. The group von Neumann algebra of $G$ is denoted by $LG$ and the canonical trace by $\tau_G$. The Kronecker delta function is denoted by $\de_{i,j}$. The use of two $\de$'s will not appear in the same place. It should be clear from context which one we are using. We let $\mathds{1}_n$ denote the $n\times n$ matrix with all entries equal to 1.

\subsection{Spin matrix model}\label{s:spin}
We consider the general spin matrix model following \cite{LP98, JPP}. Fix a finite integer $N$. Let $J_{N,m}=[N]\times [m]$ and $J_N=[N]\times \nz$. We usually do not specify the dependence on $N$ and simply write $J_m=J_{N,m}$ and $J=J_{N}$ if there is no ambiguity. We equip $J_m$ with the lexicographical order. Let $\eps:J \times J \to \{-1,1\}$ be a map satisfying $\eps(x,y)=\eps(y,x)$ and $\eps(x,x)=-1$ for all $x,y\in J$. Consider the complex unital algebra $\ax_{m}=\ax_m(N,\eps)$ generated by $(x_i(k))_{(i,k)\in J_{m}}$, where $x_i(k)$'s satisfy $x_i(k)^*=x_i(k)$ and \[
x_i(k)x_j(l)-\eps((i,k),(j,l)) x_j(l)x_i(k)=2\de_{(i,k),(j,l)}
\]
for $(i,k),(j,l)\in J_m$. It is well known that $x_i(k)$'s can be represented as tensor products of Pauli matrices. Thus $\ax_m$ can be represented as a matrix subalgebra of $M_{2^{Nm}}$. A generic element of $\ax_m$ can be written as a linear combination of words of the form
\[
x_B=x_{i_1}(k_1)\cdots x_{i_d}(k_d),
\]
where $B=\{(i_1,k_1), \cdots, (i_d,k_d)\}\subset J_m$. We say $x_B$ is a reduced word if $x_{i_r}(k_r)$'s in $x_B$ are pairwise different for $r=1,\cdots, d$. Using the commutation relation, every word $\xi$ can be written in the reduced form, denoted by $\td\xi$.  There is a canonical normalized trace $\tau_m$ on $\ax_m$ such that $\tau_m(x_B)= \de_{B,\emptyset}$ for a reduced word $x_B$.

\subsection{Pisier's method for multi-index summations}
Let $\si\in P(d)$ be a partition. In the following we need to estimate the $L_p$ norm of
$$\sum_{\vec{k}\in[m]^d: \si(\vec{k})\ge \si} x_1(k_1)\cdots x_d(k_d),
$$
where $x_i(k_i) \in \cap_{p<\8} L_p(\tau)$, and $L_p(\tau)$ is a noncommutative $L_p$ space associated to a trace $\tau$. To this end, we follow Pisier's method \cite{Pi00}. As illustrated in the proof of \cite{Pi00}*{Sublemma 3.3}, one can find $\xi_1(k_1),\cdots,\xi_d(k_d)\in LG$ such that $\tau_G(\xi_1(k_1)\cdots\xi_d(k_d))=1$ if and only if $\si(\vec{k})\ge\si$ and $\tau_G(\xi_1(k_1)\cdots\xi_d(k_d))=0$ otherwise. Here $G$ is a suitable product of free groups, $LG$ is the von Neumann algebra of $G$, and $\tau_G$ is the canonical trace on $LG$.

Let us explain this in more detail using an example. We denote by $\fz_m$ the free group with free generators $(g_i)_{i\in[m]}$. Suppose $d=6$ and $\si=\{\{1,3,5\},\{2,6\},\{4\}\}$. In this case, $G=\fz_m\times\fz_m \times\fz_m$ and for $i\in[m]$,
\begin{align*}
  \xi_1(i)=\la(g_i)^*\otimes 1\otimes 1, &\quad \xi_2(i)=1\otimes \la(g_i)^* \otimes 1,\\
  \xi_3(i)=\la(g_i)\otimes 1\otimes \la(g_i)^*,&\quad \xi_4(i)=1\otimes 1\otimes 1,\\
  \xi_5(i)=1\otimes 1\otimes \la(g_i),&\quad \xi_6(i)=1\otimes  \la(g_i)\otimes 1.
\end{align*}
Then $\tau_G(\xi_1(k_1)\cdots\xi_6(k_6))=1$ if and only if $k_1=k_3=k_5$ and $k_2=k_6$.

Now return to the general setting and consider the algebraic tensor product $LG \otimes L_p(\tau)$. Since $\tau_G\otimes \id$ extends to contractions on $L_p$, using H\"older's inequality, we have
\begin{align*}
  \Big\|\sum_{\vec{k}\in[m]^d: \si(\vec{k})\ge \si} x_1(k_1)\cdots x_d(k_d)\Big\|_p &= \Big\|\sum_{\vec{k}\in[m]^d} \tau_G(\xi_1(k_1)\cdots\xi_d(k_d))  x_1(k_1)\cdots x_d(k_d)\Big\|_p\\
  &\le \Big\|\sum_{\vec{k}\in[m]^d} \xi_1(k_1)\otimes x_1(k_1)\cdots \xi_d(k_d)\otimes x_d(k_d) \Big\|_p\\
  &\le \prod_{i=1}^d \Big\|\sum_{k_i=1}^m\xi_i(k_i)\otimes x_i(k_i)\Big\|_{pd}.
\end{align*}
If $i$ belongs to a singleton block of $\si$, then $\xi_i(k_i)=1$ and
$$\|\sum_{k_i=1}^m\xi_i(k_i)\otimes x_i(k_i)\|_{pd}=\|\sum_{k_i=1}^m x_i(k_i)\|_{pd}.$$
If $i$ does not belong to any singleton of $\si$, then it is well known that
\[
\|\sum_{k_i=1}^m \xi_i(k_i)\otimes x_i(k_i)\|_{pd} = \|\sum_{k_i=1}^m \la(g_i)\otimes x_i(k_i)\|_{pd}.
\]
By \cite{Pi00}*{Lemma 3.4}, we have for any even integer $p\ge2$,
\[
\|\sum_{k_i=1}^m \la(g_i)\otimes x_i(k_i)\|_{p} \le \frac{3\pi}{4}\max\Big\{\Big\|(\sum_{k_i} x_i(k_i)^* x_i(k_i))^{1/2}\Big\|_p, ~\Big\|(\sum_{k_i} x_i(k_i) x_i(k_i)^*)^{1/2}\Big\|_p \Big\}.
\]
We record this result as follows. Denote by $\si_{sing}$ and $\si_{ns}$ the union of singletons and the union of non-singleton blocks of $\si$ respectively. Thus we have $\#\si_{sing} + \#\si_{ns}=d$.
\begin{prop}\label{pbound}
  Let $\si\in P(d)$ be a partition and $x_i(k_i) \in L_p(\tau)$ for $\vec{k}\in [m]^d, i\in [d]$. Then for any even integer $p\ge 2$,
  \begin{align*}
  \Big\|\sum_{\vec{k}: \si(\vec{k})\ge \si} x_1(k_1)\cdots x_d(k_d) \Big\|_p &\le (\frac{3\pi}{4})^{\#\si_{ns}} \prod_{i\in \si_{sing}}\Big\|\sum_{k_i=1}^m x_i(k_i)\Big\|_{pd}\cdot\\
  & \prod_{i\in \si_{ns}} \max\Big\{\Big\|(\sum_{k_i=1}^m x_i(k_i)^* x_i(k_i))^{1/2}\Big\|_{pd}, ~\Big\|(\sum_{k_i=1}^m x_i(k_i) x_i(k_i)^*)^{1/2}\Big\|_{pd} \Big\}.
  \end{align*}
\end{prop}
This result will be used in a slightly more general setting. We may have other fixed operators $y_j$'s inside the product $x_1(k_1)\cdots x_d(k_d)$. In this case, we may simply attach $y_j$'s to their adjacent $x_i(k_i)$'s and then invoke Proposition \ref{pbound}.

\section{Construction and Wick word decomposition}
The algebra we study here can be constructed using purely operator algebraic techniques if $\max_{1\le i,j\le N} |q_{ij}|<1$ as shown in \cite{BS94}. However, we use the probabilistic approach due to Speicher \cite{Sp92,Sp93}. This is convenient for studying the analytic properties following Biane's original idea \cite{Bi}. The main result of this section is Theorem \ref{wicdec}. Although the proof is unexpectedly lengthy, the analytic properties are easy consequences of this result. As a byproduct, we also provide an alternative construction of the Fock space representation.

\subsection{Speicher's CLT and von Neumann algebra ultraproducts}\label{s:cltul}
Let $Q=(q_{ij})_{i,j=1}^N$ be a symmetric matrix where $q_{ij}=q(i,j)\in [-1,1]$. Note that we {do not specify} the values on the diagonal. Following the notation of Section \ref{s:spin}, we consider a probability space $(\Om, \pz)$ and a family of independent random variables $\eps((i,k),(j,l)):\Om\to\{-1,1\}$ for $(i,k)<(j,l)$ with distribution
\begin{equation}\label{qprob}
\pz(\eps((i,k),(j,l))=-1)=\frac{1-q(i,j)}2, \quad \pz(\eps((i,k),(j,l))=1)=\frac{1+q(i,j)}2,
\end{equation}
so that $\ez[\eps((i,k),(j,l))]=q(i,j)$. Here $(i,k),(j,l)\in [N]\times \nz$. Given $\om\in \Om$, the commutation/anticommutation relation is fixed. We understand all generators $x_{i}(k)(\om)$ depend on $\om$. Restricting $k\in[m]$ we get random $\ax_m$. Since the dependence on $\om$ should be clear from context, we will not write $\om$ in the following to simplify notation. Let $\td{x}_i(m)=\frac1{\sqrt{m}}\sum_{k=1}^m x_i(k)$. %Define $\td{\ax}_m$ to be the von Neumann algebra generated by $\{\td{x}_i(m): i=1,\cdots, N\}$. Note that $\td{\ax}_m$ is a subalgebra of $\ax_m$.

The following central limit theorem result was due to Speicher \cite{Sp93}, which is a generalization of \cite{Sp92}. We streamline Speicher's proof in the appendix for the reader's convenience. The same strategy will be used repeatedly when we prove Theorem \ref{wicdec}.
\begin{theorem}\label{clt}
  Let $\vec{i}\in [N]^s$. Then
  \[
  \lim_{m\to \8}\tau_{m}(\td x_{i_1}(m)\cdots \td x_{i_s}(m))=\de_{s\in 2\zz}\sum_{\substack{\si\in P_2(s) \\ \si\le \si(\vec{i})}}\prod_{\{r,t\}\in I(\si)}q(i(e_r),i(e_t))\quad a.s.
  \]
  Here and in what follows, we understand $\prod_{\{i,j\}\in\emptyset} q(i,j)=1$.
\end{theorem}

By Theorem \ref{clt}, we can find a full probability set $\Om_0\subset\Om$ such that the convergence holds for all $\om\in \Om_0$. Fix a free ultrafilter $\ux$ on $\nz$. By the well known ultraproduct construction of von Neumann algebras (see, e.g., \cite{BO08}*{Appendix A}), we have a finite von Neumann algebra $\ax_\ux:=\prod_{m,\ux}\ax_m$ with normal faithful tracial state $\tau_\ux=\lim_{m,\ux}\tau_m$. Put $\ax_\ux^\8=\cap_{p<\8} L_p(\ax_\ux)$. For each $\om\in\Om_0$,
\[
(\td x_i(m)(\om))^\bullet\in \ax_\ux^\8.
\]
Here and in what follows we write $(\td x_i(m)(\om))^\bullet$ for the element represented by $(\td x_i(m)(\om))_{m\in\nz}$ in the ultraproduct. We have the moment formula
\begin{equation}\label{mome}
\tau_\ux((\td x_{i_1}(m)(\om))^\bullet\cdots (\td x_{i_s}(m)(\om))^\bullet) =\de_{s\in 2\zz}\sum_{\substack{\si\in P_2(s) \\ \si\le \si(\vec{i})}}\prod_{\{r,t\}\in I(\si)}q(i(e_r),i(e_t)).
\end{equation}
It follows that
\[
\tau_\ux(|(\td x_i(m)(\om))^\bullet|^p)\le C p^p.
\]
By the uniqueness argument in \cite{Ju06}*{Section 6}, the von Neumann algebras generated by the spectral projections of $(\td x_i(m)(\om))^\bullet, i=1,\cdots,N$ for different $\om\in \Om_0$ are isomorphic. We denote by $\Ga_Q$ any von Neumann algebra in the isomorphic class with generators $(\td x_i(m)(\om))^\bullet, i=1,\cdots,N$. This algebra was introduced by Speicher \cite{Sp93} and studied in \cites{BS94,LP99}.  Note that $(\td x_i(m)(\om))^\bullet$ may be an unbounded operator, therefore may not be in $\Ga_Q$. But by our construction, it belongs to $\Ga_Q^\8:= \cap_{p<\8} L_p(\Ga_Q, \tau_\ux)$. In the following, whenever we say $(\td x_i(m)(\om))^\bullet,i=1,\cdots, N$ are generators of $\Ga_Q$, we always mean $(\td x_i(m)(\om))^\bullet \in \Ga_Q^\8$ and $\Ga_Q$ is generated by the spectral projections of $(\td x_i(m)(\om))^\bullet$'s. We call $\Ga_Q$ \textit{the} mixed $q$-Gaussian algebra, and $Q$ the structure matrix of $\Ga_Q$. Sometimes we also write $\tau_Q = \tau_\ux|_{\Ga_Q}$.

There is another way of constructing $\Ga_Q$. All the $x_i(k)$'s are actually in $L_\8(\Om; \ax_m)$ and thus $\td{x}_i(m)\in L_\8(\Om; \ax_m)$. Here the trace on $L_\8(\Om; \ax_m)$ is given by $\ez\otimes \tau_m$. By the same CLT argument as for Theorem \ref{clt}, we find the moment formula \eqref{etmom} in the limit, which is the same as \eqref{mome}. Therefore, as before, $(\td{x}_i(m))^\bullet, i=1,\cdots,N$ generates a von Neumann algebra, denoted by $\Ga_Q^a$. We call it the average model. Using the uniqueness results in \cite{Ju06}*{Section 6}, we have that $\Ga_Q^a$ is isomorphic to $\Ga_Q$. %These different models give us more flexibility to deduce some results in the following.
When we write $(\td{x}_i(m))^\bullet$, it can mean either an element in $\cap_{p<\8}L_p(\prod_{m,\ux}L_\8(\Om; \ax_m))$ or simply $(\td{x}_i(m)(\om))^\bullet$ for some $\om\in \Om_0$. It should be clear from context which one we are using. In fact, we may simply write $x_1,\cdots, x_N$ for the generators of $\Ga_Q$ if we are not concerned with the construction.

By considering different structure matrix $Q$, we can construct various examples as special cases of $\Ga_Q$. The same philosophy was used before by Lust-Piquard in \cite{LP99}.
\begin{exam}
  $\Ga_q(H)$, $q\in[-1,1]$ is fixed. If $q(i,j)=q$ all $1\le i, j\le N$, then we recover the classical $q$-Gaussian algebra $\Ga_q(H)$ where $H$ is a real Hilbert space with $\dim H=N$.
\end{exam}
\begin{exam}
  $\ast_{i=1}^n \Ga_{q_i}(H_i)$, $q_i\in[-1,1]$ is fixed for $i=1,\cdots,n$. Here $H_i$'s are real Hilbert spaces with $\dim H_i=d_i$. Let $N=d_1+\cdots+d_n$. Define $Q$ as follows. For $k=0,\cdots,n-1$ and $1\le \al,\bt\le d_{k+1}$, put
  \[
  q(\sum_{j=1}^k d_j+\al, \sum_{j=1}^k d_j+\bt)=q_{k+1} ;
   \]
   and $q(\al,\bt)=0$ otherwise. Then by the moment formula \eqref{mome}, we recover $\ast_{i=1}^n \Ga_{q_i}(H_i)$. The case $q_i=-1$ for all $i=1,\cdots, n$ was considered in \cite{JPP}.
\end{exam}
\begin{exam}
  $\overline\otimes_{i=1}^n (\Ga_{q_i}(H_i)\ast\Ga_{p_i}(K_i))$, $q_i,p_i\in[-1,1]$ are fixed. Here $H_i$'s and $K_i$'s are real Hilbert spaces with $\dim H_i=d_i, \dim K_i=d'_i$. Let $N=\sum_{i=1}^n d_i+d_i'$. For $k=0,\cdots, n-1$, define
  \[
  \td q\Big(\sum_{j=1}^{i} (d_j+d_j')+\al, \sum_{j=1}^i (d_j+d_j')+\bt\Big) =
  \begin{cases}
    q_i, \quad 1\le \al,\bt \le d_{i+1},\\
    p_i, \quad d_{i+1}+1\le \al,\bt\le d_{i+1}+d_{i+1}',\\
    0, \quad 1\le \al \le d_{i+1}<\bt\le d_{i+1}+d_{i+1}';
  \end{cases}
  \]
  and $\td q(\al,\bt)=1$ otherwise. Let $Q=(\td{q}_{\al,\bt})_{1\le \al,\bt\le N}$. By the moment formula \eqref{mome}, this model gives mixed products of $q$-Gaussian algebras. For example, consider the von Neumann algebra of the integer lattice $L(\zz^n)$. We may identify $L(\zz^n)$ with $\overline \otimes_{i=1}^n \Ga_0(\rz)$ via $\la(g_k)\mapsto x_k$, where $g_k$'s are the generators of $\zz^n$ and $x_k$'s are generators of $\overline \otimes_{i=1}^n \Ga_0(\rz)$. Alternatively, by extending $\la(g_k)\mapsto x_{2k-1}x_{2k}$, we may embed $L(\zz^n)$ into $\overline\otimes_{i=1}^n \Ga_{-1}(\rz)\ast \Ga_{-1}(\rz)$.
\end{exam}

\subsection{Wick word decomposition}
For our later development, we need an analogue of Wick word decomposition, i.e., rewriting $(\td x_{i_1}(m))^\bullet\cdots (\td x_{i_d}(m))^\bullet$ as a linear combination of Wick words (to be defined) so that we can analyze the Ornstein--Uhlenbeck semigroup easily. This procedure is conceptually clear with the help of Fock space representation because $(\td x_{i_1}(m))^\bullet\cdots (\td x_{i_d}(m))^\bullet$ belongs to $L_2(\Ga_Q)$ and $L_2(\Ga_Q)$ should coincide with the Fock space, which is spanned by Wick products; see \cite{BKS,BS94}. However, we do not know the explicit formula for the decomposition of $(\td x_{i_1}(m))^\bullet\cdots (\td x_{i_d}(m))^\bullet$ in terms of matrix models. Moreover, the known Fock space construction usually requires $\max_{i,j} |q_{ij}|<1$.

Our approach is again probabilistic. We refer the readers to Section \ref{s:nota} for the notation used in the following. By definition
\[
(\td x_{i_1}(m))^\bullet\cdots (\td x_{i_d} (m))^\bullet = \Big(\frac1{m^{d/2}} \sum_{\vec{k}\in [m]^d}x_{i_1}(k_1) \cdots x_{i_d}(k_d)\Big)^\bullet.
\]
Note that
\[
\sum_{\vec{k}\in [m]^d}x_{i_1}(k_1) \cdots x_{i_d}(k_d)=\sum_{\si\in P_{1,2}(d)}\sum_{\si(\vec{k})=\si} x_{i_1}(k_1) \cdots x_{i_d}(k_d) +\sum_{\si\in P_{r}(d)}\sum_{\si(\vec{k})=\si} x_{i_1}(k_1) \cdots x_{i_d}(k_d).
\]
We first record a simple algorithm which we will refer to later on.
\begin{prop}\label{algor}
  Let $\vec{i}\in [N]^d$, $\vec{k}\in [m]^d$, $\si(\vec{k})\le \si(\vec{i})$ and $\si(\vec{k})\in P_{1,2}(d)$. Then there is a specific algorithm to interchange $x_{i_\al}(k_\al)$'s in $x_{i_1}(k_1)\cdots x_{i_d}(k_d)$ such that
  \begin{enumerate}
    \item   $ x_{i_1}(k_1)\cdots x_{i_d}(k_d) =\eps(\vec{i},\vec{k}) x_{j_1}(l'_1)\cdots x_{j_s}(l'_s)\cdots x_{j_d}(l'_d)$,
   where $\eps(\vec{i},\vec{k})$ is a random sign resulting from interchanging $x_{i_\al}(k_\al)$'s which is given by
   \[
   \eps(\vec{i},\vec{k}) = \prod_{\{r,t\}\in I_{sp}(\si(\vec{k}))}\eps([i(e_r),k(e_r)], [i(e_t),k(e_t)]) \prod_{\{r,t\}\in I_{p}(\si(\vec{k}))}\eps([i(e_r),k(e_r)], [i(e_t),k(e_t)]);
   \]
  \item $(l_1',\cdots, l_s')$ are pairwise different and maintain their relative positions in $\vec{k}$, i.e., $(l_1',\cdots, l_s')$ is obtained from $\vec{k}$ by removing the $k_\al$'s which correspond to pair blocks;
  \item $l'_{s+1}=l'_{s+2}, \cdots, l'_{d-1}=l'_d$.
  \end{enumerate}
\end{prop}
\begin{proof}
  Since $\si(\vec{k})\in P_{1,2}(d)$, for each $k_\al$ in $\vec{k}$, there is at most one $k_\bt$ in $\vec{k}$ equal to $k_\al$. We can find the first $k_\al$ corresponding to a singleton in $\si(\vec{k})$, and move $x_{i_\al}(k_\al)$ to the beginning of the word by interchanging it with $x_{i_\bt}(k_\bt)$'s which are to the left of $x_{i_\al}(k_\al)$. Rename this $x_{i_\al}(k_\al)$ to be $x_{j_1}(l'_1)$. This process produces a product of random signs of the form $\eps((i_\al, k_\al), (i_\bt, k_\bt))$ where $k_\al$ corresponds to a singleton and $k_\bt$ corresponds to a pair block in $\si(\vec{k})$. Then we repeat this procedure for the second $k_\al$ corresponding to a singleton in $\si(\vec{k})$, and rename it $x_{j_2}(l'_2)$. We continue until all the $x_{i_\al}(k_\al)$ corresponding to singletons in $\si(\vec{k})$ are in front of the rest $x_{i_\bt}(k_\bt)$'s corresponding to pair blocks in $\si(\vec{k})$. In this way, we get $x_{j_1}(l'_1)\cdots x_{j_s}(l'_s)$ and a product of random signs. Afterwards, we rename the variable $x_{i_\al}(k_\al)$ right adjacent to $x_{j_s}(l'_s)$ to be $x_{j_{s+1}}(l'_{s+1})$. Then move the other term with the same $k_\al$ to the right of $x_{j_{s+1}}(l'_{s+1})$, call it $x_{j_{s+2}}(l'_{s+2})$. This produces a product of $\eps((i_\al,k_\al),(i_\bt,k_\bt))$ where $k_\bt$ and $k_\al$ correspond to different pair blocks. Repeat this procedure for the next pair of $k_\al$'s. After finitely many steps, the algorithm will stop and we obtain $\eps(\vec{i},\vec{k}) x_{j_1}(l'_1)\cdots x_{j_s}(l'_s)\cdots x_{j_d}(l'_d)$ with the desired three properties.
\end{proof}

We write
\begin{equation}\label{klcor}
  (l'_1,\cdots, l'_d)=(k_{\pi(1)}, \cdots, k_{\pi(d)})
\end{equation}
where $\pi$ is a permutation determined by the algorithm. Similarly,
$
(j_1,\cdots,j_d)=(i_{\pi(1)}, \cdots, i_{\pi(d)}).
$
Let
\begin{equation}\label{llsd}
  l_1=l'_1, \cdots, l_s=l'_s, l_{s+1}=l'_{s+1}=l'_{s+2}, \cdots, l_{s+u}=l'_{d-1}=l'_d.
\end{equation}
Here $s$ and $u$ are the number of singletons and pair blocks of $\si(\vec{k})$, respectively.
\begin{lemma}\label{sinpar}
  Let $\si\in P_{1,2} (d)$. Then for all $2<p<\8$ and fixed $\om\in \Om$,
  \begin{align*}
  \lim_{m\to \8}\Big\|\frac1{m^{d/2}}&\sum_{\vec{k}\in[m]^d: \si(\vec{k})=\si} x_{i_1}(k_1) \cdots x_{i_d}(k_d)
   \\
  &-  \frac1{m^{d/2}}\sum_{\vec{k}\in[m]^d: \si(\vec{k})=\si} E_{\nx_{s}(\vec{k})}[x_{i_1}(k_1) \cdots x_{i_d}(k_d)]\Big\|_{L_p(\ax_m, \tau_m)}=0.
  \end{align*}
  Here $\nx_{s}(\vec{k})$ denotes the von Neumann algebra generated by all $x_{i_\al}(k_\al)$'s, where $k_\al$ corresponds to singleton blocks in $\si(\vec{k})$.
\end{lemma}
\begin{proof}
Let $s$ and $u$ denote the number of singletons and pair blocks of $\si$, respectively. Clearly, $s+2u=d$ and there are $$m_{(s+u)}:=m(m-1)\cdots (m-s-u+1)$$
vectors $\vec{k}\in[m]^d$ with $\si(\vec{k})=\si$. Let $\vec{l}$ be given in \eqref{llsd}. $\vec{l}$ is a vector of length $s+u$. Let $\de_1,\cdots, \de_m$ be i.i.d. random selectors uniformly distributed on $\{1,2,\cdots,s+u\}$ which are independent from $L_\8(\Om;\ax_m)$. If all $l_\al$'s are pairwise different, then by independence,
\[
\ez_\de(1_{[\de_{l_1}=1]}1_{[\de_{l_2}=2]}\cdots 1_{[\de_{l_{s+u}}=s+u]}) =(s+u)^{-s-u},
\]
where $\ez_\de$ is the expectation with respect to $\de_{l_q}$'s. Define the random sets for $q=1,\cdots, s+u$,
\[
B_q=\{l_q\in [m]: \de_{l_q}=q\}.
\]
Then for each instance of $\de_{l_q}$'s, $B_q$'s are pairwise disjoint and their union equals $[m]$. By \eqref{klcor}, there is a 1-1 correspondence between $\vec{k}$ and $\vec{l}$. We may rewrite
\begin{align*}
  &\sum_{\vec{k}\in[m]^d: \si(\vec{k})=\si} x_{i_1}(k_1) \cdots x_{i_d}(k_d) =\sum_{\vec{l}\in[m]^{s+u}: \si(\vec{l})\in P_1(s+u)} x_{i_1}(k_1) \cdots x_{i_d}(k_d)\\
  &=(s+u)^{s+u}\sum_{\vec{l}: \si(\vec{l})\in P_1(s+u)}\ez_\de[1_{[\de_{l_1}=1]}1_{[\de_{l_2}=2]}\cdots 1_{[\de_{l_{s+u}}=s+u]} x_{i_1}(k_1) \cdots x_{i_d}(k_d)]\\
  &=(s+u)^{s+u}\ez_\de\big(\sum_{l_{s+u}\in B_{s+u}} \cdots \sum_{l_{1}\in B_{1}}x_{i_1}(k_1) \cdots x_{i_d}(k_d)\big)
\end{align*}
where $\si(\vec{l})\in P_1(s+u)$ amounts to saying that all $l_q$'s are pairwise different. %Assume the set $B_{s+u}=\{j_1<\cdots<j_n\}$.
For $q=s,s+1,\cdots, s+u$, let $\nx_{q}(\vec{k})$ be the von Neumann algebra generated by
\[
\{x_{j_\al}(l'_\al):  \al\le s+2(q-s)\}%\cup \{x_{i_\al}(k_\al): k_\al=l_{u+s}, l_{u+s}\le j_q\}
\]
%where $q=1,\cdots,n$, and $\nx_{\de,0}$ be generated by $\{x_{i_\al}(k_\al):  \pi(\al)\le d-2\}$.
Recall that $l_{s+u}=k_{\pi(d-1)}= k_{\pi(d)}$. Let
\[
w_{\vec{i},\vec{l}}(l_{s+u})= \sum_{l_{s+u-1}\in B_{s+u-1}}\cdots \sum_{l_1\in B_{1}} x_{i_1}(k_1) \cdots x_{i_d}(k_d).
\]
Here we only fix $l_{s+u}$ and sum over all the other indices. It is straightforward to check that
$$\{w_{\vec{i},\vec{l}}(l_{s+u}) -E_{\nx_{s+u-1}(\vec{k})} (w_{\vec{i},\vec{l}}(l_{s+u}))\}_{l_{s+u}\in B_{s+u}}$$
is a sequence of martingale differences. %Notice that $E_{\nx_{\de,q-1}} (w_{\vec{i},\vec{l}}(j_q))=E_{\nx_{\de,0}} (w_{\vec{i},\vec{l}}(j_q))$.
Using the noncommutative Burkholder--Gundy inequality \cite{PX97}, we have
\begin{align*}
&\Big\|\sum_{l_{s+u}\in B_{s+u}} \cdots \sum_{l_{1}\in B_{1}} \big( x_{i_1}(k_1) \cdots x_{i_d}(k_d)-E_{\nx_{s+u-1}(\vec{k})}[x_{i_1}(k_1) \cdots x_{i_d}(k_d)]\big)\Big\|_p\\
&=\Big\|\sum_{l_{s+u}\in B_{s+u}} \big(w_{\vec{i},\vec{l}}(l_{s+u}) -E_{\nx_{s+u-1}(\vec{k})} (w_{\vec{i},\vec{l}}(l_{s+u})) \big) \Big\|_p \\%\le C_p \Big(\sum_{q=1}^n\|w_{\vec{i},\vec{l}}(j_q) -E_{\nx_{\de,0}} (w_{\vec{i},\vec{l}}(j_q))\|_p^p\Big)^{1/p}\\
  &\le C_p\Big\|\Big(\sum_{l_{s+u}\in B_{s+u}} |w_{\vec{i},\vec{l}}(l_{s+u}) -E_{\nx_{s+u-1}(\vec{k})} (w_{\vec{i},\vec{l}}(l_{s+u})) |^2 \\
  &\qquad + |\big(w_{\vec{i},\vec{l}}(l_{s+u}) -E_{\nx_{s+u-1}(\vec{k})} (w_{\vec{i},\vec{l}}(l_{s+u})) \big)^*|^2\Big)^{1/2}\Big\|_p\\
  &=: \Psi.
\end{align*}
%Note that $\|x_{i_\al}(k_\al)\|_p\le 1$ and that the conditional expectation is a contraction on $L_p$. Since $n\le m$, we may use the triangle inequality brutally for $I$ and get
%\[
%I\le C_p m^{d/p}.
%\]
By the triangle inequality, we have
%\begin{align*}
%  E_{\nx_{\de,q-1}}\big[|w_{\vec{i},\vec{l}}(j_q) -E_{\nx_{\de,0}} (w_{\vec{i},\vec{l}}(j_q)) |^2 = \sum_{l_{s+u-1}\in B_{s+u-1}}\cdots \sum_{l_1\in B_{1}} E_{\nx_{\de,0} x_{i_1}(k_1) \cdots x_{i_d}(k_d)
%\end{align*}
\begin{align*}
  \Psi&\le C_p \sqrt{|B_{s+u}|} \sup_{l_{s+u}\in B_{s+u}} \|w_{\vec{i},\vec{l}}(l_{s+u}) -E_{\nx_{s+u-1}(\vec{k})} (w_{\vec{i},\vec{l}}(l_{s+u}))\|_p\\
  &\le 2 C_p \sqrt{|B_{s+u}|} \sup_{l_{s+u}\in B_{s+u}}\|w_{\vec{i},\vec{l}}(l_{s+u})\|_p.
\end{align*}
Recall that $k_\al=l_{\pi^{-1}(\al)}$ if $\al$ is a singleton of $\si$. In this case, $l_\bt\in B_\bt$ if and only if $k_\al\in B_{\pi^{-1}(\al)}$ where $\pi(\bt)=\al$. Replacing $p$ by a larger even integer if necessary, arguing as for Proposition \ref{pbound}, or simply adding zeros to apply Proposition \ref{pbound}, we find
\[
\|w_{\vec{i},\vec{l}}(l_{s+u})\|_p\le (\frac{3\pi}{4})^{2u}\prod_{\al\in\td\si_{sing}}\Big \| \sum_{k_\al\in B_{\pi^{-1}(\al)}} x_{i_\al}(k_\al)\Big\|_{pd}\prod_{\al\in \td\si_{ns}} m^{1/2}.
\]
Here $\td{\si}$ is obtained from $\si$ by erasing one pair block containing $\pi(d)$ so that $\#\td\si_{ns}=2(u-1)$. We mention one subtlety  here in applying Proposition \ref{pbound}. Since $l_{s+u}$ is fixed here, the term $x_{i_\al}(l_{s+u})$ is regarded to ``attach'' to its adjacent term. For instance, $x_{i_{j'}}(k_{j'})x_{i_{\al}}(l_{s+u})x_{i_{j}}(k_j)$ is regarded as a product of two terms, i.e., $[x_{i_{j'}}(k_{j'}) x_{i_{\al}}(l_{s+u})] x_{i_{j}}(k_j)$ or $x_{i_{j'}}(k_{j'}) [x_{i_{\al}}(l_{s+u}) x_{i_{j}}(k_j)]$. Using the noncommutative Khintchine inequality \cites{LP86,LPP} or Burkholder--Gundy inequality \cite{PX97}, we have for $\al\in\td\si_{sing}$,
\begin{align*}
  \Big \| \sum_{k_\al\in B_{\pi^{-1}(\al)}} x_{i_\al}(k_\al) \Big\|_{pd}&\le C_{pd}\max\Big\{\Big\| \sum_{k_\al}x_{i_\al}(k_\al)^*x_{i_\al}(k_\al)\Big\|_{pd/2}^{1/2},\Big\| \sum_{k_\al}x_{i_\al}(k_\al)x_{i_\al}(k_\al)^*\Big\|_{pd/2}^{1/2}\Big\}\\
&\le C_{pd} m^{1/2}.
\end{align*}
It follows that $\|w_{\vec{i},\vec{l}}(l_{s+u})\|_p \le C_{p, \si} m^{s/2+u-1}$ and thus $\Psi\le C_{p, \si} m^{s/2+u-1/2}$. We have shown that
\begin{equation}
  \frac{1}{m^{d/2}}\Big\| \sum_{l_{s+u}\in B_{s+u}} \cdots \sum_{l_{1}\in B_{1}} \big(x_{i_1}(k_1) \cdots x_{i_d}(k_d)-E_{\nx_{s+u-1}(\vec{k})}[x_{i_1}(k_1) \cdots x_{i_d}(k_d)]\big) \Big\|_p \le \frac{C_{p, \si}}{m^{1/2}}.
\end{equation}
Repeating the argument $u-1$ times by replacing $u$ with $u-1, u-2,\cdots, 1$, we find
\begin{align*}
  &\frac{1}{m^{q-s/2}}\Big\| \sum_{l_{q}\in B_{q}} \cdots \sum_{l_{1}\in B_{1}}\big( E_{\nx_{q}(\vec{k})} [x_{i_1}(k_1) \cdots x_{i_d}(k_d)] -E_{\nx_{q-1}(\vec{k})}[x_{i_1}(k_1) \cdots x_{i_d}(k_d)]\big) \Big\|_p\\
  & \le \frac{C_{p, \si}}{m^{1/2}},
\end{align*}
for $q=s+u-1,\cdots,s+1$. In this iteration argument, we use the same ``attaching'' procedure as described above in order to apply Proposition \ref{pbound}. By the triangle inequality, we have
\[
\frac{1}{m^{d/2}}\Big\| \sum_{l_{s+u}\in B_{s+u}} \cdots \sum_{l_{1}\in B_{1}}\big(x_{i_1}(k_1) \cdots x_{i_d}(k_d)-E_{\nx_{s}(\vec{k})}[x_{i_1}(k_1) \cdots x_{i_d}(k_d)]\big) \Big\|_p \le \frac{C_{p, \si}}{m^{1/2}}.
\]
Hence, by Jensen's inequality,
\begin{align*}
  &\Big\|\frac1{m^{d/2}}\sum_{\vec{k}: \si(\vec{k})=\si} x_{i_1}(k_1) \cdots x_{i_d}(k_d)
   -  \frac1{m^{d/2}}\sum_{\vec{k}: \si(\vec{k})=\si} E_{\nx_{s}(\vec{k})}[x_{i_1}(k_1) \cdots x_{i_d}(k_d)]\Big\|_{L_p(\ax_m, \tau_m)}\\
   &\le \frac{(s+u)^{s+u}}{m^{d/2}}\ez_\de\Big\| \sum_{l_{s+u}\in B_{s+u}} \cdots \sum_{l_{1}\in B_{1}} \big( x_{i_1}(k_1) \cdots x_{i_d}(k_d)-E_{\nx_{s}(\vec{k})}[x_{i_1}(k_1) \cdots x_{i_d}(k_d)]\big)\Big\|_p\\
   &\le \frac{C_{p, \si}}{m^{1/2}}.
\end{align*}
In the last inequality, the upper bound holds for every instance of $\de$ and thus holds for the average. The proof is complete by sending $m\to \8$.
\end{proof}
%\begin{align*}
%  &\frac1{m^{d/2}}\sum_{\si(\vec{k})=\si} x_{i_1}(k_1) \cdots x_{i_d}(k_d)= \Big(\frac1{m^{s/2}}\sum_{l_1\neq\cdots\neq l_s} \eps(\vec{i},\vec{k})x_{j_1}(l_1)\cdots x_{j_s}(l_s)\Big)\cdot \\
%  &\Big(\frac1{m^{p}}\sum_{l_{s+1}\neq l_{s+3}\neq\cdots \neq l_{d}} \eps(\vec{i},\vec{k})x_{j_{s+1}}(l_{s+1}) x_{j_{s+2}}(l_{s+1})\cdots x_{j_{d-1}}(l_d)x_{j_{d}}(l_d)\Big).
%\end{align*}
\begin{lemma}\label{remai}
  Let $\si\in P_{r} (d)$. Then for all $p<\8$ and fixed $\om\in \Om$,
  \begin{align*}
  \lim_{m\to \8}\Big\|\frac1{m^{d/2}}\sum_{\vec{k}: \si(\vec{k})=\si} x_{i_1}(k_1) \cdots x_{i_d}(k_d)
   \Big\|_{L_p(\ax_m, \tau_m)}=0,
  \end{align*}
\end{lemma}
\begin{proof}
We follow the same argument as for Lemma \ref{sinpar} and only indicate the differences. For $\si\in P_r(d)$, there is at lease one block with more than two elements. Without loss of generality, assume there is only one block in $\si$ with more than two elements. Suppose this block has, say, three elements. We list the running indices $\vec{k}$ in the sum as $\{l_1,\cdots,l_s,l_{s+1},\cdots,l_{s+u},l_{s+u+1}\}$, where there are $s$ singletons, $u$ pairs and one block with three elements in $\si$. Using the random selectors, it suffices to show that
\[
\frac1{m^{d/2}} \Big\|\sum_{l_{s+u+1}\in B_{s+u+1}}\cdots \sum_{l_1\in B_1} x_{i_1}(k_1) \cdots x_{i_d}(k_d)\Big\|_p \to 0,
\]
as $m\to\8$, where $B_1,\cdots, B_{s+u+1}$ are disjoint random sets with union $[m]$. Denote by $\nx_{s+u}(\vec{k})$ the von Neumann algebra generated by $x_{i_{\pi(\al)}}(l_\al')$ for all $\al\le s+2u$, where $\vec{l'}$ is a permutation of $\vec{k}$ so that $l_1=l'_1,\cdots, l_s=l'_s,l_{s+1}=l'_{s+1}=l'_{s+2}$, etc. Then using the noncommutative Burkholder--Gundy inequality, we can show that
\begin{align*}
  &\frac{1}{m^{d/2}}\Big\| \sum_{l_{s+u+1}\in B_{s+u+1}} \cdots \sum_{l_{1}\in B_{1}} \big( x_{i_1}(k_1) \cdots x_{i_d}(k_d)-E_{\nx_{s+u}(\vec{k})}[x_{i_1}(k_1) \cdots x_{i_d}(k_d)]\big) \Big\|_p \\
  &\le \frac{C_{p, \si} m^{s/2+u+1/2}}{m^{s/2+u+3/2}}\to 0,
\end{align*}
as $m\to \8$. It remains to show
\[
\frac{1}{m^{d/2}}\Big\| \big(\sum_{l_{s+u+1}\in B_{s+u+1}} \cdots \sum_{l_{1}\in B_{1}} E_{\nx_{s+u}(\vec{k})}[x_{i_1}(k_1) \cdots x_{i_d}(k_d)]\big) \Big\|_p \to 0.
\]
Note that
\begin{align*}
  &\frac{1}{m^{d/2}}\Big\| \big(\sum_{l_{s+u+1}\in B_{s+u+1}} \cdots \sum_{l_{1}\in B_{1}} E_{\nx_{s+u}(\vec{k})}[x_{i_1}(k_1) \cdots x_{i_d}(k_d)]\big) \Big\|_p \\
  &\le\frac1{m}\sum_{l_{s+u+1}\in B_{s+u+1}}\frac1{m^{(d-2)/2}}\Big\|  \sum_{l_{s+u}\in B_{s+u}} \cdots \sum_{l_{1}\in B_{1}} x_{i_1}(k_1) \cdots x_{i_d}(k_d) \Big\|_p.
\end{align*}
Now apply Proposition \ref{pbound} with the same ``attaching'' procedure as above. This yields
\[
\Big\|  \sum_{l_{s+u}\in B_{s+u}} \cdots \sum_{l_{1}\in B_{1}} x_{i_1}(k_1) \cdots x_{i_d}(k_d)\Big\|_p\le C_{p,\si} m^{s/2+u},
\]
which gives a decay factor and completes the proof.
\end{proof}
\begin{theorem}\label{wicdec}
Let $(\td x_j(m))^\bullet\in \cap_{p<\8} L_p(\prod_{m, \ux}L_\8 (\Om; \ax_m))$ for $j=1,\cdots,d$. Then
\[
(\td x_{i_1}(m))^\bullet\cdots (\td x_{i_d} (m))^\bullet =\sum_{\substack{\si\in P_{1,2}(d)\\\si\le \si(\vec{i}) }}w_\si(\vec{i})
\]
where the equality holds for all $\om\in \Om$ and
\begin{equation}\label{wicar}
  w_\si(\vec{i}) = \Big( \frac1{m^{d/2}}\sum_{\substack{\vec{k}\in[m]^d: \si(\vec{k})= \si}} E_{\nx_{s}(\vec{k})}[x_{i_1}(k_1) \cdots x_{i_d}(k_d)]\Big)^{\bullet}.
\end{equation}
\end{theorem}
\begin{proof}
By Lemma \ref{sinpar} and Lemma \ref{remai}, we have
\[
(\td x_{i_1}(m))^\bullet\cdots (\td x_{i_d} (m))^\bullet = \sum_{\si\in P_{1,2}(d)}\Big( \frac1{m^{d/2}}\sum_{\substack{\vec{k}\in[m]^d: \si(\vec{k})= \si}} E_{\nx_{s}(\vec{k})}[x_{i_1}(k_1) \cdots x_{i_d}(k_d)] \Big)^\bullet.
\]
By Proposition \ref{algor}, we write $(j_1,\cdots, j_d)=(i_{\pi(1)}, \cdots, i_{\pi(d)})$, $(l'_1,\cdots,l'_{d})=(k_{\pi(1)},\cdots, k_{\pi(d)})$. It follows that
\[
E_{\nx_{s}(\vec{k})}[x_{i_1}(k_1) \cdots x_{i_d}(k_d)]=\eps(\vec{i},\vec{k})x_{j_1}(l'_1)\cdots x_{j_s}(l'_s)\de_{j_{s+1},j_{s+2}}\cdots \de_{j_{d-1},j_d}.
\]
Note that $E_{\nx_{s}(\vec{k})}[x_{i_1}(k_1) \cdots x_{i_d}(k_d)]$ is nonzero only if $j_{s+1}=j_{s+2},\cdots, j_{d-1}=j_d$. Since $\si(\vec{k})=\si$, we have $\si\le \si(\vec{i})$.
\end{proof}
If $\si\le\si(\vec{i})$, we call $w_\si(\vec{i})$ defined in \eqref{wicar} the arbitrary Wick words. By Theorem \ref{wicdec},
$$
L_2(\Ga_Q)\subset L_2\dash\Span\{w_\si(\vec{i}):\vec{i}\in [N]^d, \si\in P_{1,2}(d), d\in \nz\}.
$$
Here and in what follows $L_p\dash\Span W$ means the $L_p(\tau_\ux)$ closure of linear combinations of elements in $W$. We want to identify $L_2(\Ga_Q)$ with the span of fewer Wick words. Let $\vec{i}\in [N]^s$ for $s\in \nz$. We define the special Wick words
\begin{equation}\label{wicsp}
  w(\vec{i}) = \Big(\frac1{m^{s/2}} \sum_{\vec{k}\in[m]^s: \si(\vec{k})\in P_1(s)} x_{i_1}(k_1)\cdots x_{i_s}(k_s)\Big)^{\bullet}.
\end{equation}
Let $\vec{i'}\in [N]^{s'}$. In order to understand the inner product of $w(\vec{i})$ and $w(\vec{i'})$, we first introduce some notions. Let $\{2\cdot 1, 2\cdot 2,\cdots, 2\cdot s\}$ be a multiset, each element with multiplicity 2. One can regard it as a set of cardinality $2s$ given by $[2s]=\{1,2,\cdots, s, \td1, \td2,\cdots, \td{s}\}$. Let $\si^b$ be a partition of the set $[2s]$. We call it a bipartite pair partition of $[2s]$ if
\[
\si^b =\{\{e_k,z_k\}: e_k=1,2,\cdots, s, z_k=\td1, \td2,\cdots, \td{s}\}.
\]
The set of all bipartite pair partitions is denoted by $P^b_2(2s)$. Let $\vec{i},\vec{i'}\in [m]^s$, where $\vec{i'}$ is understood as a map $\vec{i'}: \{\td1, \td2,\cdots, \td{s}\}\to [m]$. Define the concatenation operation
\begin{equation}\label{cate}
  \vec{i}\sqcup \vec{i'}=(i_1,\cdots, i_s,i'_{\td1},\cdots,i'_{\td{s}}).
\end{equation}
We denote by $\si(\vec{i}\sqcup \vec{i'})$ the partition induced by $\vec{i}$ and $\vec{i'}$ on the multiset $[2s]$. For example, $\{k,l,\td{k}\}$ are in the same block of $\si(\vec{i}\sqcup \vec{i'})$ if $i_k=i_l=i'_{\td{k}}$. Given $\si^b \in P_2^b(2s)$, define the set of bipartite crossings by
\[
I^b(\si^b)=\{\{k,l\}: 1\le k,l\le s, e_k<e_l, z_l>z_k\}.
\]
Recall that $\lge w(\vec{i}), w(\vec{i'})\rge = \tau_\ux[w(\vec{i'})^* w(\vec{i})]$.
\begin{prop}\label{swinn}
  Let $w(\vec{i})$ and $w(\vec{i'})$ be special Wick words. Then there exists a full probability set $\Om_0\subset \Om$ such that for all $\om\in \Om_0$,
  \begin{equation*}
    \lge w(\vec{i}), w(\vec{i'})\rge =
    \begin{cases}
      \sum_{\substack{\si^b\in P_2^b(2s) \\ \si^b\le \si(\vec{i}\sqcup \vec{i'})}}\prod_{\{r,t\}\in I^b(\si^b)}q(i(e_r),i(e_t)) , \text{ if } \{i_1,\cdots, i_s\}=\{i'_1,\cdots, i'_{s'}\}, \\
      0, \text{ otherwise}.
    \end{cases}
  \end{equation*}
  where $\{i_1,\cdots, i_s\}=\{i'_1,\cdots, i'_{s'}\}$ means that $\vec{i}$ and $\vec{i'}$ are equal as multisets, i.e., both the elements and their multiplicities are the same.
\end{prop}
\begin{proof}
  We follow the same argument as for Theorem \ref{clt}. By definition,
  \[
  \lge w(\vec{i}), w(\vec{i'})\rge = \lim_{m,\ux}\frac{1}{m^{(s+s')/2}}\sum_{\substack{\vec{k},\vec{k'}: \si(\vec{k})\in P_1(s),\\ \si(\vec{k'})\in P_1(s')}} \tau_m[x_{i'_{s'}}(k'_{s'})\cdots x_{i'_{1}}(k'_{1})x_{i_1}(k_1)\cdots x_{i_s}(k_s)].
  \]
  Since all $k_\al$'s are pairwise different, $\tau_m[x_{i'_{s'}}(k'_{s'})\cdots x_{i'_{1}}(k'_{1})x_{i_1}(k_1)\cdots x_{i_s}(k_s)]=0$ unless $s=s'$. Moreover, every $x_{i_\al}(k_\al)$ has to be the same as exact one $x_{i'_\bt}(k'_\bt)$ in order to contribute to the sum. This implies that $\vec{i}$ and $\vec{i'}$ are equal as multisets. We rewrite
  \begin{align*}
  &\sum_{\substack{\vec{k},\vec{k'}: \si(\vec{k})\in P_1(s),\\ \si(\vec{k'})\in P_1(s')}} \tau_m[x_{i'_{s'}}(k'_{s'})\cdots x_{i'_{1}}(k'_{1})x_{i_1}(k_1)\cdots x_{i_s}(k_s)] \\
  &= \sum_{\substack{ \si(\vec{k}),\si(\vec{k'})\in P_1(s)}} \tau_m[x_{i'_{s}}(k'_{s})\cdots x_{i'_{1}}(k'_{1})x_{i_1}(k_1)\cdots x_{i_s}(k_s)]\\
  &= \sum_{\substack{\si^b\in P_2^b(2s)\\ \si^b\le \si(\vec{i}\sqcup \vec{i'})}}\sum_{\si(\vec{k}\sqcup \vec{k'})=\si^b}\tau_m[x_{i'_{s}}(k'_{s})\cdots x_{i'_{1}}(k'_{1})x_{i_1}(k_1)\cdots x_{i_s}(k_s)]
  \end{align*}
  If $\{r,t\}\in I^b(\si^b)$, then we have to switch $x_{i({e_r})}(k({e_r}))$ and $x_{i({e_t})}(k({e_t}))$ to cancel the corresponding $x_{i({z_r})}(k({z_r}))$ and $x_{i({z_t})}(k({z_t}))$ terms. It follows that
\[
\tau_m[x_{i'_{s}}(k'_{s})\cdots x_{i'_{1}}(k'_{1})x_{i_1}(k_1)\cdots x_{i_s}(k_s)] = \prod_{\{r,t\}\in I^b(\si^b)} \eps([i(e_r),k(e_r)], [i(e_t),k(e_t)]).
\]
Since $\vec{k}\in P_1(s)$, by independence, we have
\begin{align*}
  &\frac1{m^s}\sum_{\si(\vec{k}\sqcup \vec{k'})=\si^b}\ez \tau_m[x_{i'_{s}}(k'_{s})\cdots x_{i'_{1}}(k'_{1}) x_{i_1}(k_1)\cdots x_{i_s}(k_s)] \\
  &=\frac{m(m-1)\cdots (m-s+1)}{m^s} \prod_{\{r,t\}\in I^b(\si^b)} q(i(e_r), i(e_t)).
\end{align*}
Hence, if $\vec{i}=\vec{i'}$ as multisets, then
\[
\ez \lge w(\vec{i}), w(\vec{i'})\rge = \sum_{\substack{\si^b\in P_2^b(2s) \\ \si^b\le \si(\vec{i}\sqcup \vec{i'})}}\prod_{\{r,t\}\in I^b(\si^b)}q(i(e_r),i(e_t)).
\]
To show almost sure convergence, let
\[
X_m=\frac1{m^s}\sum_{\substack{\si^b\in P_2^b(2s)\\ \si^b\le \si(\vec{i}\sqcup \vec{i'})}}\sum_{\si(\vec{k}\sqcup \vec{k'})=\si^b}\tau_m[x_{i'_{s}}(k'_{s})\cdots x_{i'_{1}}(k'_{1})x_{i_1}(k_1)\cdots x_{i_s}(k_s)].
\]
Since $\pz(\om: |X_m-\ez X_m|>\eta)\le \Var(X_m)/\eta^2$, by the Borel--Cantelli lemma, it suffices to show that $\sum_{m=1}^\8 \Var(X_m)<\8$. But
\[
\Var(X_m)=\frac1{m^{2s}}\sum_{\si^b,\pi^b\in P_2^b(2s)}\sum_{\substack{\si(\vec{k}\sqcup \vec{k'})=\si^b\\ \si(\vec{\ell}\sqcup \vec{\ell'})=\pi^b}} V_{\vec{k},\vec{\ell}},
\]
where
\begin{align*}
  V_{\vec{k},\vec{\ell}} =&\ez(\tau_m[x_{i'_{s}}(k'_{s})\cdots x_{i'_{1}}(k'_{1})x_{i_1}(k_1)\cdots x_{i_s}(k_s)] \tau_m[x_{i'_{s}}(\ell'_{s})\cdots x_{i'_{1}}(\ell'_{1})x_{i_1}(\ell_1)\cdots x_{i_s}(\ell_s)])\\
  &-\ez(\tau_m[x_{i'_{s}}(k'_{s})\cdots x_{i'_{1}}(k'_{1})x_{i_1}(k_1)\cdots x_{i_s}(k_s)])\ez(\tau_m[x_{i'_{s}}(\ell'_{s})\cdots x_{i'_{1}}(\ell'_{1})x_{i_1}(\ell_1)\cdots x_{i_s}(\ell_s)])\\
  =&\ez\big[\prod_{\{r,t\}\in I^b(\si^b)} \eps([i(e_r),k(e_r)], [i(e_t),k(e_t)]) \prod_{\{r',t'\}\in I^b(\pi^b)} \eps([i(e_{r'}),\ell(e_{r'})], [i(e_{t'}),\ell(e_{t'})])\big] \\
  &-  \prod_{\{r,t\}\in I^b(\si^b)} q(i(e_r), i(e_t))  \prod_{\{r',t'\}\in I^b(\pi^b)} q(i(e_{r'}), i(e_{t'})).
\end{align*}
By independence, $V_{\vec{k},\vec{\ell}}$ is nonzero only if at least one pair $\{r,t\}\in I^b(\si^b)$ and $\{r',t'\}\in I^b(\pi^b)$ such that $\{k(e_r),k(e_t)\}=\{\ell(e_{r'}), \ell(e_{t'})\}$. In this case
\[
\#\{\vec{k},\vec{k'},\vec{\ell},\vec{\ell'}: \si(\vec{k}\sqcup \vec{k'})=\si^b, \si(\vec{\ell}\sqcup \vec{\ell'})=\pi^b\}\le m^sm^{s-2}=m^{2s-2}.
\]
Since $V_{\vec{k},\vec{\ell}}$ is uniformly bounded and $C(s):=[\#P_2^b(2s)]^2$ is independent from $m$,  we have
\[
\sum_{m=1}^\8\Var(X_m)\le\sum_{m=1}^\8 \frac{C(s)}{m^2}<\8. \qedhere
\]
\end{proof}
Recall the notation $I_p(\si)$ and $I_{sp}(\si)$ from Section \ref{s:nota}. For $\vec{i}\in[N]^d$ and $\si\in P_{1,2}(d)$ with $\si\le \si(\vec{i})$, put
\begin{equation}\label{qppr}
  f_\si(\vec{i})=\prod_{\{r,t\}\in I_p(\si)}q(i(e_r),i(e_t)) \prod_{\{r,t\}\in I_{sp}(\si)}q(i(e_r),i(e_t)),
\end{equation}
with the convention that the product over an empty index set is 1.
\begin{prop}\label{inner}
  Let $\si\in P_{1,2}(d)$ and $\si'\in P_{1,2}(d')$ be partitions. Let $w_\si(\vec{i})$ and $w_{\si'}(\vec{i'})$ be arbitrary Wick words defined in \eqref{wicar}. Then for almost all $\om\in \Om$,
  \begin{equation}\label{wwinn}
    \lge w_{\si}(\vec{i}), w_{\si'}(\vec{i'})\rge =
    \begin{cases}
      \lge f_\si(\vec{i})w(\vec{i}_{np}), f_{\si'}(\vec{i'})w(\vec{i'}_{np})\rge , \text{ if }\si\le \si(\vec{i}), \si'\le \si(\vec{i'}), \\
      0, \text{ otherwise}.
    \end{cases}
  \end{equation}
  Here $\vec{i}_{np}$ is the vector obtained by removing coordinates in $\vec{i}$ which correspond to the pair blocks of $\si$.
\end{prop}
\begin{exam}
  Suppose $\vec{i}=(2, 4,7, 4, 7)$ and $\si=\{\{1\},\{2\},\{4\},\{3,5\}\}$. Then $\vec{i}_{np}=(2,4,4)$.
\end{exam}
\begin{proof}[Proof of Proposition \ref{inner}]
  By definition,
  \[
  \lge w_{\si}(\vec{i}), w_{\si'}(\vec{i'})\rge= \lim_{m,\ux}\frac{1}{m^{(d+d')/2}}\sum_{\substack{\vec{k},\vec{k'}: \si(\vec{k})=\si,\\ \si(\vec{k'})=\si'}} \tau_m[x_{i'_{d'}}(k'_{d'})\cdots x_{i'_{1}}(k'_{1})x_{i_1}(k_1)\cdots x_{i_d}(k_d)].
  \]
  Note that $w_\si(\vec{i})$ is nonzero only if $\si\le \si(\vec{i})$. Then $k_\al=k_\bt$ implies $i_\al=i_\bt$. By \eqref{nonp}, we may assume that in $x_{i'_{d'}}(k'_{d'})\cdots x_{i'_{1}}(k'_{1})x_{i_1}(k_1)\cdots x_{i_d}(k_d)$ if $k_\al=k_\bt$ for $\al\neq \bt$, then $k'_\ga\neq k_\al$ for all $\ga\in[d']$. In other words,  $k_\al\neq k'_\ga$ for all $\al\in[d]$ and $\ga\in[d']$ if both of them belong to pair blocks. Applying Proposition \ref{algor} to $x_{i_1}(k_1)\cdots x_{i_d}(k_d)$ and $x_{i'_{1}}(k'_{1})\cdots x_{i'_{d'}}(k'_{d'})$, we find
  \begin{align*}
  &\tau_m[x_{i'_{d'}}(k'_{d'})\cdots x_{i'_{1}}(k'_{1})x_{i_1}(k_1)\cdots x_{i_d}(k_d)]\\
  &= \eps(\vec{i},\vec{k})\eps(\vec{i'},\vec{k'}) \tau_{m}[x_{j'_{s'}}(\ell'_{s'})\cdots x_{j'_1}(\ell_1')x_{j_1}(\ell_1) \cdots x_{j_s}(\ell_s)],
  \end{align*}
  where $\vec{\ell}\in P_1(s), \vec{\ell'}\in P_1(s')$, $\vec{j}=\vec{i}_{np}$, $\vec{j'}=\vec{i'}_{np}$, and  $\eps(\vec{i},\vec{k})\eps(\vec{i'},\vec{k'})$ is given in Proposition \ref{algor}. By independence, we have
  \begin{align*}
  &\ez\tau_m[x_{i'_{d'}}(k'_{d'})\cdots x_{i'_{1}}(k'_{1})x_{i_1}(k_1)\cdots x_{i_d}(k_d)] \\
  = &\prod_{\{r,t\}\in I_p(\si)}q(i(e_r),i(e_t)) \prod_{\{r,t\}\in I_{sp}(\si)}q(i(e_r),i(e_t)) \prod_{\{r',t'\}\in I_p(\si')}q(i(e_{r'}),i(e_{t'}))\cdot \\
  &\prod_{\{r',t'\}\in I_{sp}(\si')}q(i(e_{r'}),i(e_{t'}))
  \ez\tau_{m}[x_{j'_{s'}}(\ell'_{s'})\cdots x_{j'_1}(\ell_1')x_{j_1}(\ell_1) \cdots x_{j_s}(\ell_s)].
  \end{align*}
  As shown in Proposition \ref{swinn}, $\tau_{m}[x_{j'_{s'}}(\ell'_{s'})\cdots x_{j'_1}(\ell_1')x_{j_1}(\ell_1) \cdots x_{j_s}(\ell_s)]$ is zero if $\vec{\ell}$ and $\vec{\ell'}$ are not equal as multisets, and there is nothing more to prove. Assume $\vec{\ell}$ and $\vec{\ell'}$ are equal. Let $u$ and $u'$ be the number of pair blocks in $\si$ and $\si'$. By Proposition \ref{swinn}, we find
  \begin{align*}
    \ez \lge w_{\si}(\vec{i}), w_{\si'}(\vec{i'})\rge =& \lim_{m,\ux} \frac{m\cdots (m-s+1)}{m^{s}}\cdot \frac{(m-s)\cdots (m-s-u-u'+1)}{m^{u+u'}}\cdot \\
    &f_\si(\vec{i})f_{\si'}(\vec{i'}) \ez\tau_{m}[x_{j'_{s'}}(\ell'_{s'})\cdots x_{j'_1}(\ell_1')x_{j_1}(\ell_1) \cdots x_{j_s}(\ell_s)]\\
    =&  f_{\si}(\vec{i})f_{\si'}(\vec{i'})\ez\lge w(\vec{i}_{np}),  w(\vec{i'}_{np})\rge.
  \end{align*}
  The almost sure convergence follows from the same argument as for Proposition \ref{swinn} using the Borel--Cantelli lemma and independence.
\end{proof}
In the two proofs above, the Borel--Cantelli lemma may be avoided if we use the average model $\Ga^a_Q$; see Section \ref{s:cltul}. Note that for $\vec{i}\in [m]^s$, $w_\si(\vec{i})=w(\vec{i})$ for any $\si \in P_1(s)$.
\begin{cor}\label{ident}
  Let $\si\le\si(\vec{i})$. We have $w_\si(\vec{i})=f_\si(\vec{i})w(\vec{i}_{np})$ for almost all $\om\in \Om$.
\end{cor}
\begin{proof}
  Since $\tau_\ux$ is faithful on $\Ga_Q$, it suffices to show
  \[
  \tau_\ux[(w_\si(\vec{i})-f_\si(\vec{i})w(\vec{i}_{np}))^* (w_\si(\vec{i})-f_\si(\vec{i})w(\vec{i}_{np}))]=0.
  \]
  But by Proposition \ref{inner}, we have
  \[
  \lge w(\vec{i}_{np}), w_\si(\vec{i})\rge= f_{\si}(\vec{i})\lge w(\vec{i}_{np}), w(\vec{i}_{np})\rge.
  \]
  From here the claim follows by linearity.
\end{proof}
This result yields the following identification
\[
L_2\dash\Span\{w_\si(\vec{i}):\vec{i}\in [N]^d, \si\in P_{1,2}(d), d\in \zz_+\}= L_2\dash\Span\{w(\vec{i}): \vec{i}\in [N]^s, s\in \zz_+\}
\]
with the inner product relation given by \eqref{wwinn}.
\begin{prop}\label{isom1}
 $L_2(\Ga_Q)= L_2\dash\Span\{w(\vec{i}): \vec{i}\in [N]^s, s\in \zz_+\}$.
\end{prop}
\begin{proof}
  Write $H_w:=L_2\dash\Span\{w(\vec{i}): \vec{i}\in [N]^s, s\in \zz_+\}$. By Theorem \ref{wicdec} and Corollary \ref{ident}, $L_2(\Ga_Q)\subset H_w$. It remains to show that $H_w\subset L_2(\Ga_Q)$. We proceed by induction on the length $s$ of special Wick words $w(\vec{i})$. First observe that if $\si(\vec{i})\in P_1(s)$, then the only partition $\si\le \si(\vec{i})$ is $\si(\vec{i})$ itself. In this case, by Theorem \ref{wicdec}, we have
  \begin{equation}\label{sibel}
      (\td x_{i_1}(m))^\bullet\cdots (\td x_{i_s} (m))^\bullet= w_{\si(\vec{i})}(\vec{i})=w(\vec{i})\in L_2(\Ga_Q)
  \end{equation}
  since every $(\td x_{i_1}(m))^\bullet$ is in $\cap_{p<\8} L_p(\Ga_Q)$.
  If $s=1$,
  \[
  w(\vec{i})=\big(\frac1{\sqrt{m}}\sum_{k_1=1}^m x_{i_1}(k_1)\big)^\bullet \in L_2(\Ga_Q)
  \]
  by definition. If $s=2$ and $i_1\neq i_2$, then $w(\vec{i})\in L_2(\Ga_Q)$ by $\eqref{sibel}$. If $i_1=i_2$, using Theorem \ref{wicdec}, we find
  \[
  (\td x_{i_1}(m))^\bullet (\td x_{i_2} (m))^\bullet =w_{\si(\vec{i})}(\vec{i}) +w_{\si_0}(\vec{i})=1+w_{\si_0}(\vec{i}),
  \]
  where $\si_0\in P_1(2)$. It follows that $w(\vec{i})=w_{\si_0}(\vec{i})\in L_2(\Ga_Q)$. Suppose $w(\vec{i})\in L_2(\Ga_Q)$ for all $\vec{i}$ with $|\vec{i}|<s$. Consider $\vec{i}\in [N]^s$. We know $w(\vec{i})\in L_2(\Ga_Q)$ if $\si(\vec{i})\in P_1(s)$. If $\si(\vec{i})\notin P_1(s)$, by Theorem \ref{wicdec}, we have
  \begin{equation}\label{iindu}
  (\td x_{i_1}(m))^\bullet\cdots (\td x_{i_s} (m))^\bullet = w_{\si_0}(\vec{i}) +\sum_{\substack{\si\in P_{1,2}(s)\\ \si\le \si(\vec{i}),\si\notin P_1(s)}}w_{\si}(\vec{i}),
  \end{equation}
  where $\si_0\in P_1(s)$. By Corollary \ref{ident}, we have $w_{\si}(\vec{i})=f_\si(\vec{i})w(\vec{i}_{np})$, and $\vec{i}_{np}$ is a vector of dimension at most $s-2$. By the induction hypothesis, $w_{\si}(\vec{i})\in L_2(\Ga_Q)$ for $\si\notin P_1(s)$. We deduce from \eqref{iindu} that $w(\vec{i})=w_{\si_0}(\vec{i})\in L_2(\Ga_Q)$.
\end{proof}
\subsection{Fock spaces and mixed $q$-commutation relations}\label{s:fock}
As consequences of the work in the previous section, we can describe the Fock space and creation/annihilation operators associated to $\Ga_Q$. Given a vector $\vec{i}$, we denote by $|\vec{i}|$ the number of nonzero coordinates in $\vec{i}$. Let $H^s_Q=\Span\{w(\vec{i}): |\vec{i}|=s\}$. We define the mixed Fock space by
\begin{equation}
  \fx_Q=\bigoplus_{s=0}^\8 H^s_Q.
\end{equation}
Clearly, $\fx_Q=L_2\dash\Span\{w(\vec{i}): \vec{i}\in [N]^s, s\in \zz_+\}$, which can be further identified with $L_2(\Ga_Q)$ by Proposition \ref{isom1}.
\begin{prop}
  Let $x_j=(\frac1{\sqrt{m}} \sum_{k=1}^m x_j(k))^\bullet \in \Ga_Q^\8, j=1,\cdots, N$ be generators of $\Ga_Q$ and $w(\vec{i})\in H^s_Q$. Then
  \[
  x_j w(\vec{i}) = w(j\sqcup \vec{i}) +\sum_{l=1}^s \de_{j,i_l} w(\vec{i}- i_l) \prod_{r=1}^{l-1}q(i_r,i_{l}) .
  \]
  Here $j\sqcup \vec{i}=(j,i_1,\cdots,i_s)\in[N]^{s+1}$ is the concatenation operation defined in \eqref{cate}, $\vec{i}- i_l=(i_1,\cdots,i_{l-1},i_{l+1},\cdots,i_s)$ with $|\vec{i}- i_l|=s-1$, and we understand  the product over empty index set is 1. Therefore,
  \[
  x_j=\sum_{s=0}^\8 P_{s+1} x_j P_s +\sum_{s=1}^\8 P_{s-1}x_j P_s,
  \]
  where $P_s: \fx_Q\to H^s_Q$ is the orthogonal projection.
\end{prop}
\begin{proof}
By definition,
\begin{align*}
  x_j w(\vec{i})= &\Big(\frac1{\sqrt{m}} \sum_{k_0=1}^m x_j(k_0) \frac1{m^{s/2}} \sum_{\vec{k}\in[m]^s: \si(\vec{k})\in P_1(s)} x_{i_1}(k_1)\cdots x_{i_s}(k_s) \Big)^\bullet\\
  =&\Big(\frac1{{m}^{(s+1)/2}}\sum_{k_0\sqcup \vec{k}\in[m]^{s+1}: \si(k_0\sqcup\vec{k})\in P_1(s+1)} x_{j}(k_0) x_{i_1}(k_1)\cdots x_{i_s}(k_s) \Big)^\bullet \\
  &+\sum_{l=1}^s \Big(\frac1{m^{(s+1)/2}} \sum_{\substack{k_0\sqcup\vec{k}\in[m]^{s+1}: k_0=k_l\\ \si(\vec{k})\in P_1(s)}} x_j(k_0)x_{i_1}(k_1)\cdots x_{i_s}(k_s)\Big)^\bullet.
\end{align*}
The first term in the above equation is clearly the special Wick word $w(j\sqcup \vec{i})$. To understand the second one, we define $\si_l \in P_{1,2}(s+1)$ by $\si_l=\si(k_0\sqcup \vec{k})$ for $k_0=k_l$ and $\vec{k}\in P_1(s)$, i.e.,
\[
\si_l=\{\{1,l+1\},\{2\}, \cdots, \{l\},\{l+2\}, \cdots, \{s+1\} \}.
\]
Using (the proof of) Lemma \ref{sinpar}, we deduce that the arbitrary Wick word
\[
\Big(\frac1{m^{(s+1)/2}} \sum_{\substack{k_0\sqcup\vec{k}\in[m]^{s+1}: k_0=k_l\\ \si(\vec{k})\in P_1(s)}} x_j(k_0)x_{i_1}(k_1)\cdots x_{i_s}(k_s)\Big)^\bullet = w_{\si_l}(j\sqcup\vec{i}).
\]
Note that $w_{\si_l}(j\sqcup\vec{i})$ is nonzero only if $\si_l\le \si(j\sqcup\vec{i})$ or equivalently $j=i_l$. Using \eqref{qppr} and Corollary \ref{ident}, we find
\[
w_{\si_l}(j\sqcup\vec{i}) =\de_{j,i_l} \prod_{r=1}^{l-1}q(i_r,i_{l}) w(\vec{i}- i_l).\qedhere
\]
\end{proof}
Define operators $c_j$ and $a_j$ acting on $\fx_Q$ by
\begin{equation}\label{crann}
c_jw(\vec{i})=w(j\sqcup \vec{i}), \quad a_j w(\vec{i})= \sum_{l=1}^s\de_{j,i_l} w(\vec{i}- i_l) \prod_{r=1}^{l-1}q(i_r,i_{l}) .
\end{equation}
Clearly $x_j=c_j+a_j$, $c_j=\sum_{s=0}^\8 P_{s+1} x_j P_s$ and $a_j=\sum_{s=1}^\8 P_{s-1}x_j P_s$. Since $x_j=x_j^*$, we have $c_j^*=a_j$. We call $c_j$ and $a_j$ the creation and annihilation operators respectively for $j=1,\cdots,N$. The following result is simply a recapitulation; see e.g. \cite{BO08} for more about QWEP $\mathrm{C}^*$-algebras.
\begin{cor}
  Let $\td\Ga_Q$ be the von Neumann algebra generated by (the spectral projections of) $c_j+a_j$ for $j=1,\cdots, N$. Then $\Ga_Q=\td\Ga_Q$. In particular, $\td\Ga_Q$ is QWEP.
\end{cor}
\begin{prop}
For $j,k=1,\cdots, N$,  $c_j$ and $c_j^*$ satisfy the mixed $q$-commutation relation
  \begin{equation}\label{qcomm}
      c_k^* c_j-q(j,k)c_jc^*_k=\de_{j,k}1.
  \end{equation}
\end{prop}
\begin{proof}
  Let $w(\vec{i})\in H_Q^s$. Then,
  \[
  c_j^*c_j w(\vec{i})=c_j^* w(j\sqcup \vec{i}) =w(\vec{i})+\sum_{l=1}^{s}\de_{j,i_l}w(j\sqcup (\vec{i}-i_l))q(j,i_l)\prod_{r=1}^{l-1}q(i_r,i_l).
  \]
  But
  \[
  c_jc_j^*w(\vec{i})=\sum_{l=1}^s \de_{j,i_l}w(j\sqcup (\vec{i}-i_l))\prod_{r=1}^{l-1} q(i_r, i_l).
  \]
  Hence $c_j^* c_jw(\vec{i})-q(j,j)c_jc^*_jw(\vec{i})=w(\vec{i})$. If $j\neq k$, then
  \[
  c_k^*c_j w(\vec{i})=c_k^* w(j\sqcup \vec{i}) =\sum_{l=1}^{s}\de_{k,i_l}w(j\sqcup (\vec{i}-i_l))q(j,i_l)\prod_{r=1}^{l-1}q(i_r,i_l),
  \]
  and
  \[
  c_jc_k^*w(\vec{i})=\sum_{l=1}^s \de_{k,i_l}w(j\sqcup (\vec{i}-i_l))\prod_{r=1}^{l-1} q(i_r, i_l).
  \]
  Hence $c_k^* c_jw(\vec{i})-q(j,k)c_jc^*_kw(\vec{i})=0$.
\end{proof}
\begin{rem}\label{rempr}
  The Fock space representation was studied in more general setting by Bo\.zejko and Speicher in \cite{BS94}. Let $(e_i)$ be an o.n.b. of a Hilbert space $H$. One can construct the Fock space $\fx_Q(H)$ following \cite{BS94,LP99}. Let $\Om$ be the vacuum state and $W$ be the Wick product, i.e.,
  \[
  W(e_{i_1}\otimes \cdots \otimes e_{i_s})\Om = e_{i_1}\otimes \cdots \otimes e_{i_s}.
  \]
  The Wick product was studied in detail in \cite{Kr00}. Suppose $\vec{i}\in [N]^s$ and $\vec{j}\in [N]^{s'}$. We have
  \[
  \lge w(\vec{i}), w(\vec{j})\rge =\lge e_{i_1}\otimes \cdots \otimes e_{i_s}, e_{j_1}\otimes \cdots \otimes e_{j_{s'}}\rge,
  \]
  where the left side is given by Proposition \ref{swinn} and the right side is understood as the inner product in $\fx_Q(H)$; see \cites{BS94, LP99}. Our argument shows that one can alternatively implement \eqref{qcomm} and construct the Fock space using the probabilistic approach (Speicher's CLT) and the von Neumann algebra ultraproduct. If $\sup_{ij} q_{ij}< 1$ and $\vec{i}\in[N]^s$, we can identify our special Wick words $w(\vec{i})$ with $W(e_{i_1}\otimes \cdots \otimes e_{i_s})$. Thus we can also identify $H_Q^s$ with $H^{\otimes s}$. This identification will play an important role when we study the operator algebraic properties of $\Ga_Q$ in later parts of this paper.
\end{rem}

\subsection{The Ornstein--Uhlenbeck semigroup on $\Ga_Q$}\label{ousg}
Let $T_t^m$ be the Ornstein--Uhlenbeck semigroup acting on $\ax_m$; see \cite{Bi}*{Section 2.1}. $T_t^m$ is given by
\[
T_t^m x_{i_1}(k_1)\cdots x_{i_d}(k_d)= e^{-td}x_{i_1}(k_1)\cdots x_{i_d}(k_d)
\]
if $\si(\vec{k})\in P_1(d)$. Let us first recall an elementary fact.
\begin{lemma}\label{l128}
  Let $(\nx,\tau)$ be noncommutative $W^*$ probability space, where $\nx$ is a von Neumann algebra and $\tau$ is a normal faithful tracial state. Let $T:\nx\to\nx$ be a $*$-preserving linear normal map with pre-adjoint map $T_*: \nx_*\to \nx_*$. Suppose $T$ is self-adjoint on $L_2(\nx,\tau)$. Then $T=T_*|_{\nx}$.
\end{lemma}
\begin{proof}
Let $x,y\in \nx$. Denote the dual pairing between $x_*\in \nx_*$ and $x$ by $( x_*, x)$, which can be implemented by $( x_*, x)=\tau(x_* x)$. Since $\nx\subset \nx_*=L_1(\nx,\tau)$,
\[
( Tx,y) = \tau((Tx)y)=\lge Tx, y^*\rge_{L_2(\nx,\tau)} =\tau(x (Ty))=( x, Ty) = ( T_*x,y).\qedhere
\]
\end{proof}
Let $(T_t^m)_*: L_1(\ax_m)\to L_1(\ax_m)$ be the pre-adjoint map of $T_t$. By Lemma \ref{l128}, it coincides with $T_t$ on $\ax_m$. Let $\prod_{m,\ux}L_1(\ax_m)$ be the ultraproduct of Banach spaces $L_1(\ax_m)$. Recall that $\ax_\ux=\prod_{m,\ux}\ax_m$ is the von Neumann algebra ultraproduct in Section \ref{s:cltul}. Note that we have the canonical inclusion $L_1(\ax_\ux,\tau_\ux)\subset \prod_{m,\ux}L_1(\ax_m,\tau_m)$. Let $((T_t^m)_*)^\bullet$ be the usual ultraproduct of $(T_t^m)_*$. If $(x_m)^\bullet \in \ax_\ux$, then
\[
((T_t^m)_*)^\bullet(x_m)^\bullet=(T_t^m x_m)^\bullet \in \ax_\ux
\]
because $\sup_{m}\|T_t^m x_m\|\le \sup_{m}\|x_m\|<\8$. Hence, $((T_t^m)_*)^\bullet$ leaves $\ax_\ux$ invariant. We have checked the commutative diagram
\[%\xymatrixcolsep{7pc}
\xymatrix{
\ax_\ux\ar@{^{(}->}[r] \ar[d]^{((T_t^m)_*)^\bullet|_{\ax_\ux}} & L_1(\ax_\ux,\tau_\ux)\ar@{^{(}->}[r] \ar[d]^{((T_t^m)_*)^\bullet|_{L_1(\ax_\ux,\tau_\ux)}}   & \prod_{m,\ux}L_1(\ax_m,\tau_m) \ar[d]^{((T_t^m)_*)^\bullet}\\
\ax_\ux\ar@{^{(}->}[r] &  L_1(\ax_\ux,\tau_\ux) \ar@{^{(}->}[r] & \prod_{m,\ux}L_1(\ax_m,\tau_m).
}
\]
%By \cite{Ray}, $\big(\prod_{m,\ux}L_1(\ax_m)\big)^*$ is a von Neumann algebra which contains the finite von Neumann algebra $\ax_\ux$.
We define $T_t=[((T_t^m)_*)^\bullet|_{L_1(\ax_\ux,\tau_\ux)}]^*$. Then by construction, $T_t:\ax_\ux\to \ax_\ux$ is a normal unital completely positive map which is self-adjoint on $L_2(\ax_\ux,\tau_\ux)$. By Lemma \ref{l128} again, $T_t$ coincides with $((T_t^m)_*)^\bullet$ on $\ax_\ux$ and thus on $L_2(\ax_\ux,\tau_\ux)$. Since $\Ga_Q\subset \ax_\ux$ is a von Neumann subalgebra, $L_2(\Ga_Q)\subset L_2(\ax_\ux)\subset L_1(\ax_\ux)$. Therefore, for $\vec{i}\in[N]^s$, $w(\vec{i})\in L_2(\Ga_Q)$,
$$
T_t w(\vec{i}) = \Big( \frac1{m^{s/2}}\sum_{\vec{k}: \si(\vec{k})\in P_1(s)} e^{-t s} x_{i_1}(k_1)\cdots x_{i_s}(k_s)\Big)^{\bullet}=e^{-t s}w(\vec{i})\in L_2(\Ga_Q).
$$
Since $L_2(\Ga_Q)\cap \ax_\ux =\Ga_Q$, $T_t$ leaves $\Ga_Q$ invariant. Also, we see that $(T_t)_{t\ge0}$ is a strongly continuous semigroup in $L_2(\Ga_Q)$. Note that in general $(T_t)_{t\ge0}$ may not be a point $\si$-weakly continuous semigroup in $t$, hence may not extend to a strongly continuous semigroup on $L_2(\ax_\ux)$. By Theorem \ref{wicdec},
\[
T_t\big[(\td x_{i_1}(m))^\bullet\cdots (\td x_{i_d} (m))^\bullet \big] =\sum_{\substack{\si\in P_{1,2}(d)\\\si\le \si(\vec{i}) }} e^{-t |\si_{sing}|}w_\si(\vec{i}) = \sum_{\substack{\si\in P_{1,2}(d)\\\si\le \si(\vec{i}) }} e^{-t |\vec{i}_{np}|}f_\si(\vec{i})w(\vec{i}_{np}),
\]
where $|\si_{sing}|$ is the number of singletons in $\si$, and $|\vec{i}_{np}|$ is the dimension of $\vec{i}_{np}$. $f_\si(\vec{i})$ and $\vec{i}_{np}$ are defined in \eqref{qppr} and Proposition \ref{inner}. The generator of $T_t$ is the number operator, denoted by $A$.

\section{Analytic properties}\label{anaprop}
Our goal of this section is to prove some analytic properties for $\Ga_Q$. This will be done via a limit procedure, as was used in \cite{Bi,JPP} for proving hypercontractivity.

\subsection{Hypercontractivity}
It was proved by Biane \cite{Bi}*{Theorem 5} that the Ornstein--Uhlenbeck semigroup acting on $\ax_m=\ax_m(N,\eps)$ is hypercontractive.
\begin{theorem}\label{hypem}
  Let $1\le p,r<\8$. Then for every $\om\in \Om$
  \[
  \|T_t^m\|_{L_p\to L_r}=1 \quad \mbox{if and only if}\quad e^{-2t}\le \frac{p-1}{r-1}.
  \]
\end{theorem}
With the hard work done in the previous section, it is very easy to prove the following result.
\begin{theorem}\label{hypeu}
  Let $T_t$ be the Ornstein--Uhlenbeck semigroup on $\Ga_Q$ for arbitrary $N\times N$ symmetric matrix $Q$ with entries in $[-1,1]$. Then for $1\le p,r<\8$,
  \[
  \|T_t\|_{L_p\to L_r}=1 \quad \mbox{if and only if}\quad e^{-2t}\le \frac{p-1}{r-1}.
  \]
\end{theorem}
\begin{proof}
  The ``only if'' part follows verbatim Biane's argument \cite{Bi}*{p. 461}. For the converse, since the special Wick words span $L_p(\Ga_Q)$, it suffices to prove that if $e^{-2t}\le \frac{p-1}{r-1}$ then
  \[
  \Big\|T_t\Big(\sum_{\vec{i}} \al_{\vec{i}}w(\vec{i})\Big)\Big\|_r \le \Big\|\sum_{\vec{i}} \al_{\vec{i}}w(\vec{i})\Big\|_p
  \]
  where $\sum_{\vec{i}} \al_{\vec{i}}w(\vec{i})$ is a finite linear combination of special Wick words. But by Theorem \ref{hypem},
   \begin{align*}
   &\Big\|T^m_t\big( \sum_{\vec{i}} \frac{\al_{\vec{i}}}{m^{d(\vec{i})/2}} \sum_{\vec{k}\in[m]^{d(\vec{i})}: \si(\vec{k})\in P_1(d(\vec{i}))} x_{i_1}(k_1)\cdots x_{i_{d(\vec{i})}}(k_{d(\vec{i})})\big) \Big\|_r \\
   &\le \Big\|\sum_{\vec{i}} \frac{\al_{\vec{i}}}{m^{d(\vec{i})/2}} \sum_{\vec{k}\in[m]^{d(\vec{i})}: \si(\vec{k})\in P_1(d(\vec{i}))} x_{i_1}(k_1)\cdots x_{i_{d(\vec{i})}}(k_{d(\vec{i})})\Big\|_p.
   \end{align*}
  Since there is a canonical inclusion $L_p(\Ga_Q)\subset \prod_{m,\ux}L_p(\ax,\tau_m)$, we have
  \begin{align*}
  \Big\|\sum_{\vec{i}} \al_{\vec{i}}w(\vec{i})\Big\|_p = \lim_{m,\ux} \Big\|\sum_{\vec{i}} \frac{\al_{\vec{i}}}{m^{d(\vec{i})/2}} \sum_{\vec{k}\in[m]^{d(\vec{i})}: \si(\vec{k})\in P_1(d(\vec{i}))} x_{i_1}(k_1)\cdots x_{i_{d(\vec{i})}}(k_{d(\vec{i})})\Big\|_p.
  \end{align*}
  Similarly,
  \begin{align*}
    &\Big\|T_t\Big(\sum_{\vec{i}} \al_{\vec{i}}w(\vec{i})\Big)\Big\|_r = \Big\|\sum_{\vec{i}} \al_{\vec{i}}e^{-t|\vec{i}|}w(\vec{i})\Big\|_r\\
    &= \lim_{m,\ux} \Big\|\sum_{\vec{i}} \frac{\al_{\vec{i}}e^{-t|\vec{i}|}}{m^{d(\vec{i})/2}} \sum_{\vec{k}\in[m]^{d(\vec{i})}: \si(\vec{k})\in P_1(d(\vec{i}))} x_{i_1}(k_1)\cdots x_{i_{d(\vec{i})}}(k_{d(\vec{i})})\Big\|_r\\
    &=\lim_{m,\ux}\Big\|T^m_t\big( \sum_{\vec{i}} \frac{\al_{\vec{i}}}{m^{d(\vec{i})/2}} \sum_{\vec{k}\in[m]^{d(\vec{i})}: \si(\vec{k})\in P_1(d(\vec{i}))} x_{i_1}(k_1)\cdots x_{i_{d(\vec{i})}}(k_{d(\vec{i})})\big) \Big\|_r.
  \end{align*}
  The assertion follows immediately.
\end{proof}
This result in particular implies the hypercontractivity results for $\Ga_q(H)$ due to Biane \cite{Bi} and for $\ast \Ga_{-1}(\rz^n)$ obtained in \cite{JPP}. See also \cite{Kr05} for another generalization with the braid relation. Using the standard argument \cite{Bi}, the log-Sobolev inequality follows from optimal hypercontractivity bounds. Recall that $A$ is the number operator associated to $\Ga_Q$.
\begin{cor}[log-Sobolev inequality]
  For any finite linear combination of special Wick words $f=\sum_{\vec{i}}\al_{\vec{i}}w(\vec{i})$,
  \[
  \tau_{Q}(|f|^2\ln |f|^2)-\|f\|_2^2\ln\|f\|_2^2\le 2 \tau_{Q}(f Af^*).
  \]
\end{cor}

\subsection{Derivations}\label{s:deri}
Given the $N\times N$ matrix $Q=(q_{ij})$, we define a $2N\times 2N$ matrix
\[
 Q' =Q\otimes {1~1\choose 1~1}=\left(\begin{array}{cc}
Q&Q\\
Q&Q
\end{array}
\right).
\]
Recall that $\ax_{m}(N,\eps)$ is the spin matrix system with $Nm$ generators as in Section \ref{s:spin} and $J_{N,m}=[N]\times [m]$. We can extend the function $\eps$ to define on $J_{2N,m}\times J_{2N,m}$ as follows:
\[
\eps'((i,k),(j,l))=
\begin{cases}
  \eps((i,k),(j,l)),\quad (i,k),(j,l)\in J_{N,m},\\
  \eps((i-N, k), (j, l)),\quad 1\le j<N+1\le i\le 2N,\\
  \eps((i, k), (j-N, l)),\quad 1\le i<N+1\le j\le 2N,\\
  \eps((i-N, k), (j-N, l)),\quad N+1\le i,j\le 2N.
\end{cases}
\]
In other words, $\eps' = \eps\otimes{1~ 1 \choose 1~ 1}$. We may write $\eps$ for $\eps'$ without causing any ambiguity. Now we define a linear map
\begin{align}
  \de: \ax_m (N,\eps)&\to \ax_m(2N,\eps')\label{derim}\\
x_{i_1}(k_1)x_{i_2}(k_2)\cdots x_{i_n}(k_n) &\mapsto \sum_{\al=1}^n x_{i_1}(k_1)\cdots x_{i_{\al-1}}(k_{\al-1}) x_{i_\al+N}(k_\al)x_{i_{\al+1}}(k_{\al+1})\cdots x_{i_n}(k_n),\nonumber
\end{align}
where $x_{i_1}(k_1)x_{i_2}(k_2)\cdots x_{i_n}(k_n) $ is assumed to be in the reduced form. See also \cite{LP99}. It is easy to see that $\de$ is $*$-preserving.
%where we assume $(i_\al, k_\al)$'s are pairwise different, i.e., $x_{i_1}(k_1)x_{i_2}(k_2)\cdots x_{i_s}(k_s)$ is in the reduced form.
\begin{lemma}
 $\de$ is a derivation, i.e., $\de(\xi\eta)=\de(\xi)\eta+\xi\de(\eta)$ for two words $\xi$ and $\eta$.
\end{lemma}
\begin{proof}
The assertion follows from the fact that $\de$ is the derivative of certain one parameter group of automorphisms; see \cites{LP98, LP99, ELP}. We provide a direct elementary proof here. Note that if $\xi$ and $\eta$ are two reduced words with no common generators, the derivation property follows easily from \eqref{derim}. It remains to verify the derivation property when $\xi$ and $\eta$ have common generators.
\begin{comment}
Indeed, for any generator $x_i(k)$, since
$$
\eps((i+N,k),(i,k))=\eps((i,k),(i,k))=-1,
$$
we have
$$
\de(x_i(k))x_i(k)+x_i(k)\de(x_i(k)) =x_{i+N}(k)x_i(k) +x_i(k)x_{i+N}(k)=0,
$$
which coincides with the definition $\de(x_i(k)^2)=\de(1)=0$. Now we proceed by induction.
\end{comment}
Let $\xi=x_{i_1}(k_1)\cdots x_{i_n}(k_n)\in \ax_m(N,\eps)$ be a reduced word and $a$ an arbitrary generator. Assume $a=x_{i(\al_0)}(k(\al_0))$ and write the reduced form of $a\xi$ as $\td{a\xi}$. Then
\[
\td{a\xi}=\eps((i_1,k_1),(i_{\al_0}, k_{\al_0}))\cdots \eps((i_{\al_0-1},k_{\al_0-1}),(i_{\al_0}, k_{\al_0})) x_{i_1}(k_1)\cdots \check x_{i_{\al_0}}(k_{\al_0})\cdots x_{i_n}(k_n),
\]
 where $\check x$ means the generator $x$ is omitted in the expression. We have
 \begin{align*}
     \de(\td{a\xi})=&\eps((i_1,k_1),(i_{\al_0}, k_{\al_0}))\cdots \eps((i_{\al_0-1},k_{\al_0-1}),(i_{\al_0}, k_{\al_0}))\\
     &\sum_{\al=1,\al\neq \al_0}^n x_{i_1}(k_1)\cdots x_{i_{\al-1}}(k_{\al-1}) x_{i_\al+N}(k_\al)x_{i_{\al+1}}(k_{\al+1})\cdots x_{i_n}(k_n).
 \end{align*}
  Here we understand that if $\al=\al_0-1$ then $i_{\al+1}$ is actually $i_{\al+2}$ because $x_{i_{\al_0}}(k_{\al_0})$ is omitted. Similar remark applies when $\al=\al_0+1$ and we will follow this convention to ease notation in this proof. On the other hand,
  \begin{align*}
  &\de(a)\xi+a\de(\xi)\\
  =& x_{i_{\al_0}+N}(k_{\al_0})x_{i_1}(k_1)\cdots x_{i_n}(k_n)+ a x_{i_1}(k_1)\cdots x_{i_{\al_0}+N}(k_{\al_0})\cdots x_{i_n}\\
  &+\prod_{j=1}^{\al_0-1} \eps((i_j,k_j),(i_{\al_0}, k_{\al_0}))
     \sum_{\al=1,\al\neq \al_0}^n x_{i_1}(k_1)\cdots x_{i_{\al-1}}(k_{\al-1}) x_{i_\al+N}(k_\al)x_{i_{\al+1}}(k_{\al+1})\cdots x_{i_n}(k_n)\\
  =&\prod_{j=1}^{\al_0-1} \eps((i_j,k_j),(i_{\al_0}, k_{\al_0})) \sum_{\al=1,\al\neq \al_0}^n x_{i_1}(k_1)\cdots x_{i_{\al-1}}(k_{\al-1}) x_{i_\al+N}(k_\al)x_{i_{\al+1}}(k_{\al+1})\cdots x_{i_n}(k_n).
  \end{align*}
Here we used the commutation relation given by $\eps$ in both equalities. Hence,
\begin{equation}\label{basede}
\de(\td{a\xi})=\de(a)\xi+a\de(\xi).
\end{equation}
Now assume $\de(\td{\eta\xi})=\de(\eta)\xi+\eta\de(\xi)$ where both $\xi$ and $\eta$ are reduced words and the generators of $\eta$ are all in $\xi$, i.e., $\eta$ is a sub-word of $\xi$. We want to show that $\de(\td{a\eta \xi})=\de(\td{a\eta})\xi+a\eta\de(\xi)$ where $a$ is a generator. Note that $\td{a\eta\xi}=\td{a\td{\eta\xi}}$. By \eqref{basede} and the induction hypothesis,
  \begin{align*}
  \de(\td{a\td{\eta\xi}})&=\de(a) \eta\xi+a\de(\td{\eta\xi} )\\
  &= \de(a)\eta\xi+a\de(\eta)\xi+a\eta\de(\xi)=\de(\td{a\eta})\xi+a\eta\de(\xi).
  \end{align*}
The derivation property is verified when $\eta$ is a sub-word of $\xi$. For arbitrary reduced words $\xi$ and $\eta$, using the commutation relation we can write $\eta=\eta_1\eta_2$ so that $\eta_1$ and $\xi$ have no common generators and the generators of $\eta_2$ are in $\xi$. Then
\begin{align*}
\de(\td{\eta\xi})& = \de(\eta_1\td{\eta_2\xi})=\de(\eta_1)\eta_2\xi+\eta_1\de(\td{\eta_2\xi})=\de(\eta_1)\eta_2\xi+\eta_1\de(\eta_2)\xi+\eta_1\eta_2\de(\xi)\\
&=\de(\eta_1\eta_2)\xi + \eta_1\eta_2\de(\xi)=\de(\eta)\xi+\eta\de(\xi).
\end{align*}
The proof is complete.
\end{proof}
This lemma implies in particular that $\de(\xi)$ can be defined by \eqref{derim} and equals $\de(\td{\xi})$ even if $\xi$ is not a reduced word. We will simply write $\de(\xi)$ for any word $\xi$ in the following. If we denote by $A^m$ the number operator associated to the spin system $\ax_m(N,\eps)$, the gradient form is defined as
\[
\Ga^m(f,g) = \frac12[A^m(f^*)g+f^*A^m(g)-A^m(f^*g)]
\]
for $f,g\in \ax_m(N,\eps)$. The superscript $m$ is used to distinguish the operators from their counterparts defined for the limiting algebra $\Ga_Q$. We may simply omit this superscript if there is no ambiguity.
\begin{lemma}\label{exde}
  Let $f,g\in \ax_m(N,\eps)$. Then
  \[
\Ga(f,g)=E(\de(f)^*\de(g))
  \]
  where $E: \ax_m(2N,\eps)\to \ax_m(N,\eps)$ is the conditional expectation satisfying
  $$E(x_B)=\de_{B\cap(\{N+1,\cdots,2N\}\times[m] ), \emptyset}x_B
  $$
  for a reduced word $x_B$.
\end{lemma}
\begin{proof}
By linearity, it suffices to check $\Ga(f,g)=E(\de(f)^*\de(g))$ if $f$ and $g$ are reduced words in $\ax_m(N,\eps)$. Let $X_B=x_{i_1}(k_1)\cdots x_{i_n}(k_n), X_C=x_{j_1}(l_1)\cdots x_{j_s}(l_s)$ be two reduced words where $B,C\subset [N]\times [m]$ consist of $(i_\al,k_\al)$ and $(j_\bt,l_\bt)$ respectively. By the derivation property \eqref{derim},
\begin{align*}
  &E(\de(X_B)^*\de(X_C))\\
  &=\sum_{\al=1}^n\sum_{\bt=1}^s E(x_{i_n}(k_n)\cdots x_{i_\al+N}(k_\al)\cdots x_{i_1}(k_1) x_{j_1}(l_1)\cdots x_{j_\bt+N}(l_\bt)\cdots x_{j_s}(l_s)).
\end{align*}
We claim that the only nonzero terms in the above sum are those with $(i_\al,k_\al)=(j_\bt, l_\bt)$. Indeed, the conditional expectation simply computes the trace of generators with subscript greater than $N$ in the reduced form of
$$x_{i_n}(k_n)\cdots x_{i_\al+N}(k_\al)\cdots x_{i_1}(k_1) x_{j_1}(l_1)\cdots x_{j_\bt+N}(l_\bt)\cdots x_{j_s}(l_s).$$
Thus $x_{i_\al+N}(k_\al)$ and $x_{j_\bt+N}(l_\bt)$ have to be the same to cancel out in order to contribute to the sum. It follows that
\begin{equation}\label{ede1}
\begin{aligned}
&E(\de(X_B)^*\de(X_C))\\
&=\sum_{\al,\bt: (i_\al,k_\al)=(j_\bt,l_\bt)} E(x_{i_n}(k_n)\cdots x_{i_\al+N}(k_\al)\cdots x_{i_1}(k_1) x_{j_1}(l_1)\cdots x_{j_\bt+N}(l_\bt)\cdots x_{j_s}(l_s))\\
&=\sum_{\al,\bt: (i_\al,k_\al)=(j_\bt,l_\bt)} x_{i_n}(k_n)\cdots x_{i_\al}(k_\al)\cdots x_{i_1}(k_1) x_{j_1}(l_1)\cdots x_{j_\bt}(l_\bt)\cdots x_{j_s}(l_s).
\end{aligned}
\end{equation}
Here we used the extended commutation relation on $\ax(2N,\eps)$ given by $\eps$ in the last equality. Since $X_B$ and $X_C$ are reduced, given $(i_\al,k_\al)\in B$ there is at most one $(j_\bt,l_\bt)\in C$ such that they are equal, and vice versa. We see that there are $|B\cap C|$ terms in the sum of \eqref{ede1}. Hence, we find
\[
E(\de(X_B)^*\de(X_C)) = |B\cap C| X_{B}^*X_C.
\]
On the other hand,
\begin{align*}
  &\Ga(X_B,X_C)\\
  &=\frac12(A(X_B^*)X_C+X_B^*A(X_C)-A(X_B^*X_C))\\
  &=\frac12(|B|+|C|-|B\triangle C|)x_{i_n}(k_n)\cdots x_{i_\al}(k_\al)\cdots x_{i_1}(k_1) x_{j_1}(l_1)\cdots x_{j_\bt}(l_\bt)\cdots x_{j_s}(l_s).
\end{align*}
Note that we have the same word here as the summand of \eqref{ede1}. Since $2|B\cap C|=|B|+|C|-|B\triangle C|$, we must have
\[
\Ga(X_B,X_C)=|B\cap C|X_B^*X_C =E(\de(X_B)^*\de(X_C)).\qedhere
\]
\end{proof}
Let $w(\vec{i})\in H^s_Q$ be a special Wick word with length $s\in \zz_+$. We define a linear map $\de: H^s_Q\to H^s_{Q'}$
\begin{equation}\label{deriu}
  \de(w(\vec{i}))=(\frac1{m^{s/2}}\sum_{\vec{k}: \si(\vec{k})\in P_1(s)} \de^m [x_{i_1}(k_1)\cdots x_{i_s}(k_s)])^\bullet
\end{equation}
where $\de^m$ is the derivation defined in \eqref{derim}. Here we used Remark \ref{cltrk} implicitly. Note that $\de^m$ is bounded when acting on words with fixed length $s$ although it is not uniformly (in $m$) bounded on $\ax_m$. Hence $\de=(\de^m)^\bullet$ is well-defined on $H_Q^s$. Since $L_2(\Ga_Q)=\oplus_{s=0}^\8 H^s_Q$, we can define $\de$ on each $H^s_Q$ by \eqref{deriu}. By definition, $\de$ is densely defined on $\fx_Q=L_2(\Ga_Q)$ and $\Dom(\de)= \Dom(A)$ can be identified with the linear span of special Wick words with finite length, where $A$ is the number operator on $L_2(\Ga_Q)$. Since each $w(\vec{i})$ is actually in $\Ga_Q^\8$, $\de(w(\vec{i}))$ is in $\Ga_{Q'}^\8$.
\begin{prop}
  $\de:L_2(\Ga_Q)\to L_2(\Ga_{Q'})$ is a closed derivation.
\end{prop}
\begin{proof}
  Let $P_s: L_2(\Ga_Q)\to H^s_Q$ and $P_s': L_2(\Ga_{Q'})\to H_{Q'}^s$ be the orthogonal projections. Suppose $x_n\in \Dom(\de)$, $\lim_{n\to\8}\|x_n\|_2=0$ and $\lim_{n\to\8}\|\de(x_n)-y\|_2=0$. Then $P_s x_n\to 0$ for each $s\in\zz_+$. It follows that
  \[
  P_s'\de(x_n)=\de(P_s(x_n))\to0 \quad \mbox{as}\quad n\to\8.
  \]
  But $P_s'\de(x_n)\to P_s' y$, we find $P_s' y=0$ and thus $y=0$. Hence $\de$ is closed. The derivation property follows from the definition \eqref{deriu}, \eqref{derim} and Remark \ref{cltrk}.
\end{proof}
Denote by $\ax_\ux(N)$ the von Neumann algebra ultraproduct of $\ax_m(N)$. Then $E=(E^m)^\bullet: \ax_\ux(2N)\to \ax_\ux(N)$ is the canonical conditional expectation, where $E^m: \ax_m(2N)\to\ax_m(N)$ is given in Lemma \ref{exde}. Since $\Ga_Q\subset\ax_\ux(N)$ as a von Neumann subalgebra, there is trace preserving conditional expectation $E:\Ga_{Q'}\to\Ga_Q$ which extends to contractions on $L_p$ for $1\le p<\8$. Recall that $\Ga(\cdot,\cdot)$ is the gradient form associated with the number operator $A$ on $\Ga_Q$.
\begin{prop}\label{exd2}
  Let $f,g\in \Dom(\de)$. Then
  \[
  \Ga(f,g)=E(\de(f)^*\de(g)).
  \]
  \end{prop}
\begin{proof}
By linearity, it suffices to check the claim for $f=w(\vec{i})$ and $g=w(\vec{i'})$. By the construction of conditional expectation, the proof of Lemma \ref{exde} and the construction of the Ornstein--Uhlenbeck semigroup on $L_2(\Ga_Q)$ in Section \ref{ousg}, we have
\begin{align*}
  &E[\de(w(\vec{i}))^*\de(w(\vec{i'}))]\\
=& \Big(\frac{1}{m^{(d+d')/2}} E^m \big(\de^m\big( \sum_{\vec{k}: \si(\vec{k})\in P_1(d)} x_{i_1}(k_1)\cdots x_{i_{d}}(k_{d})\big)^* \de^m\big(\sum_{\vec{k'}: \si(\vec{k'})\in P_1(d')} x_{i'_1}(k'_1)\cdots x_{i'_{d'}}(k'_{d'})\big)\Big)^\bullet\\
= &\frac12[A(w(\vec{i}^*)w(\vec{i'})+w(\vec{i})^*A(w(\vec{i'})) - A(w(\vec{i})^*w(\vec{i'}))]=\Ga(w(\vec{i}), w(\vec{i'})),
\end{align*}
where $A$ is the number operator on $\Ga_Q$.
\end{proof}
\subsection{Riesz transforms}
Lust-Piquard showed the boundedness of Riesz transforms for the general spin system in \cite{LP98}. Let $T\in\ax_{m}(N,\eps)$ with $\tau_m(T)=0$. Recall that the Riesz transforms $R_j(T)=D_j (A^m)^{-1/2}(T)$ where $D_j$ is the annihilation operator and $A^m=\sum_{j=1}^{Nm} D_j^* D_j$ is the number operator for the spin system $\ax_{m}(N,\eps)$. By Lemma 3.2 and Proposition 1.3 in \cite{LP98}, we have for $1<p<\8$,
\begin{equation}\label{ritr1}
  \td K_{p'}^{-1}\|T\|_p\le \|\sum_{j=1}^{Nm} P_jR_j(T)\|_p\le \td K_p\|T\|_p,
\end{equation}
where $\td K_p=O(p^3/(p-1)^{3/2})$, $\frac1p+\frac1{p'}=1$, and $P_j$ is a certain tensor of Pauli matrices in the general spin system; see \cite{LP98}*{Def. 2.1}. It is known (see \cite{LP99}*{p. 547}) that $\|\sum_{j=1}^{Nm} P_jR_j(T)\|_p=\|\de^m (A^m)^{-1/2}(T)\|_p$, where $\de^m$ is the derivation defined in \eqref{derim}. By considering $T=(A^m)^{1/2}f$, \eqref{ritr1} can be rewritten as
\begin{equation}\label{ritr2}
  \td K_{p'}^{-1}\|(A^m)^{1/2}f\|_p\le \|\de^m(f)\|_p\le \td K_p\|(A^m)^{1/2}f\|_p,
\end{equation}
Now it is easy to recover Lust-Piquard's main result in \cite{LP99}. Recall that $A$ is the number operator on $\Ga_Q$.
\begin{theorem}[Lust-Piquard]\label{lprie}
  Let $1<p<\8$ and $\frac1p+\frac1{p'}=1$. Let $\de$ be defined by \eqref{deriu}. Then for any $f\in \Dom(\de)$,
  \[
  \td K_{p'}^{-1}\|A^{1/2}f\|_p\le \|\de(f)\|_p\le \td K_p\|A^{1/2}f\|_p,
  \]
  where $\td K_p=O(p^3/(p-1)^{3/2})$.
\end{theorem}
\begin{proof}
We may assume without loss of generality that $f=\sum_{\vec{i}}\al_{\vec{i}}w(\vec{i})$ is a finite linear combination of special Wick words. Write $w(\vec{i})=(X(\vec{i},m))^\bullet$. Then,
\[
\|\de(f)\|_p=\Big\|\sum_{\vec{i}}\al_{\vec{i}} \de(w(\vec{i}))\Big\|_p =\lim_{m,\ux} \Big\|\sum_{\vec{i}}\al_{\vec{i}} \de^m(X(\vec{i},m))\Big\|_p.
\]
Similarly,
\[
\|A^{1/2}f\|_p = \lim_{m,\ux} \Big\|\sum_{\vec{i}}\al_{\vec{i}} \sqrt{|\vec{i}|}X(\vec{i},m)\Big\|_p =\lim_{m,\ux}\|(A^m)^{1/2}f\|_p.
\]
The assertion follows from \eqref{ritr2} with a limiting procedure.
\end{proof}
In fact, we can give more precise estimates using the gradient form. Let
\[
G_p = L_p\dash\Span \{w(\vec{i}): \vec{i}\in [2N]^s, s\in \nz, 1\le i_k\le N \mbox{ for all but at most one } k \}.
\]
Since $L_p(\Ga_Q)\subset G_p \subset L_p(\Ga_{Q'})$, we have $E:G_p \to L_p(\Ga_Q)$ given by the restriction of the conditional expectation $E:\Ga_{Q'}\to \Ga_Q$.
If $f\in \Ga_{Q'}$, we define $\|f\|_{L_p^c(E)} = \|E(f^*f)^{1/2}\|_p$ and $\|f\|_{L_p^r(E)}=\|f^*\|_{L_p^c(E)}$. The conditional $L_p(\Ga_{Q'})$ space is
\[
L_p^{rc}(E)=\begin{cases}
  L_p^r(E)+L_p^c(E),\mbox{ if } 1\le p\le 2,\\
  L_p^r(E)\cap L_p^c(E), \mbox{ if } 2\le p\le \8.
\end{cases}
\]
Define $G_p^{r}$ (resp. $G_p^c$) as the space of $G_p$ with the norm inherited from $L_p^{r}(E)$ (resp. $L_p^c(E)$). Now we follow \cite{JMP14} to derive a Khintchine type inequality. First, since $E:\Ga_{Q'}\to \Ga_Q$ extends to contractions on $L_p$ for $1\le p<\8$, we have for $f\in L_p(\Ga_{Q'})$ and $2\le p<\8$,
\begin{equation}\label{khin0}
  \max\{\|E(f^*f)^{1/2}\|_p, \|E(ff^*)^{1/2}\|_p\} \le \|f\|_{L_p}.
\end{equation}
This means that $L_p(\Ga_{Q'})\subset L_p^{rc}(E)$ contractively for $2\le p<\8$.
\begin{lemma}\label{khin}
  Let $E:G_p\to L_p(\Ga_Q)$ be as above. Then for $2\le p<\8$,
  \[
  \|f\|_{G_p}\le C\sqrt{p} \max\{\|E(f^*f)^{1/2}\|_p, \|E(ff^*)^{1/2}\|_p\} \le C\sqrt{p}\|f\|_{G_p},
  \]
  and for $1<p\le 2$,
  \[
  \|f\|_{G_p}\le \inf_{\substack{f=g+h\\ g\in G_p^c,h\in G_p^r}} \{\|E(g^*g)^{1/2}\|_p +\|E(hh^*)^{1/2}\|_p\} \le C\sqrt{\frac{p}{p-1}} \|f\|_{G_p}.
  \]
\end{lemma}
\begin{proof}
  Let $2\le p<\8$. The right inequality is a special case of \eqref{khin0}. For the left inequality, let $n\in \nz$ and $\vec{i}\in [2N]^s$. For $j=1,\cdots, n$, define
\[
 \phi_j: [2N]^s\to [2N]^{ns}, \quad \phi_j(\vec{i})=\vec{0}\sqcup \cdots \vec{0} \sqcup\vec{i}\sqcup\vec{0}\cdots \sqcup\vec{0},
\]
 where $\vec{i}$ occurs in the $j$-th position. Put $\td\pi_j(w(\vec{i}))=w(\phi_j(\vec{i}))$, where $w(\phi_j(\vec{i}))$ is the special Wick word associated to $\phi_j(\vec{i})$. Define
\[
\pi_n: \Ga_{Q'}\to \Ga_{ Q'\otimes \mathds{1}_n},\quad \pi_n(w(\vec{i}))=\frac1{\sqrt{n}}\sum_{j=1}^n w(\phi_j(\vec{i})).
\]
Here $\mathds{1}_n$ is the $n\times n$ matrix with all entries equal to 1. $\pi_n$ extends to a trace preserving $*$-homomorphism. Alternatively, one may define $\pi_n$ via the second quantization functor as in \cite{LP99}. It is crucial to observe that $\td{\pi}_j(w(\vec{i})), j=1,\cdots, n$ are fully independent over $\Ga_{Q\otimes \mathds{1}_n}$ (see \cite{JZ11}) if $w(\vec{i})\in G_p$. This can be checked from the definition of $E_{\Ga_Q}:\Ga_{Q'}\to \Ga_Q$.
% It seems we don't need to assume they are in G_p. True in general?
We may assume that $f$ is a finite linear combination of special Wick words in $G_p$. Using the noncommutative Rosenthal inequality \cites{JX08, JZ11}, we have
  \begin{align*}
  \|\pi_n(f)\|_p\le&  \frac{Cp}{\sqrt{n}}\big(\sum_{j=1}^n\|\td\pi_j(f)\|_p^p\big)^{1/p}
  + \frac{C\sqrt{p}}{\sqrt{n}} \max\Big\{\Big\|\Big(\sum_{j=1}^n  E[\td\pi_j(f)^*\td\pi_j(f)]\Big)^{1/2}\Big\|_p,  \\
  &\Big\|\Big(\sum_{j=1}^n  E[\td\pi_j(f)\td\pi_j(f)^*]\Big)^{1/2}\Big\|_p\Big\}.
  \end{align*}
  Here we have extended the conditional expectation $E: \Ga_{Q'}\to \Ga_Q$ to $ E: \Ga_{ Q'\otimes \mathds{1}_n}\to \Ga_{Q\otimes \mathds{1}_n}$. Note that $E[\td\pi_j(f)^*\td\pi_j(f)] =\td{\pi}_j[E(f^*f)]$ and that $\|\td\pi_j(f)\|_p=\|f\|_p$. Sending $n\to \8$ for $2<p<\8$, we have
  \begin{equation}\label{gpmax}
  \|f\|_{G_p}\le C\sqrt{p} \max\{\|E(f^*f)^{1/2}\|_p, \|E(ff^*)^{1/2}\|_p\}.
  \end{equation}
  For the case $1<p<2$, we argue by duality. Define the orthogonal projection $P: L_2^{rc}(E)\to G_2\cap L_2^{rc}(E)$. By orthogonality, for $g\in \Ga_{Q'}$,
  $$E(g^*g) =E(Pg^*Pg)+E(P^\perp g^*P^\perp g)\ge E(Pg^*Pg).
  $$
  Similarly, $E(gg^*)\ge E(PgPg^*)$. Since
  \begin{align*}
  \max\{\|E(Pg^*Pg)^{1/2}\|_p, \|E(PgPg^*)^{1/2}\|_p\}
  \le \max\{\|E(g^*g)^{1/2}\|_p, \|E(gg^*)^{1/2}\|_p\},
  \end{align*}
 we deduce from \eqref{gpmax} that $P$ extends to a bounded projection with norm
  \[
  \|P: L_p^{rc}(E)\to L_p(\Ga_{Q'})\|\le C\sqrt{p}
  \]
   for $2\le p<\8$. For $1<p\le 2$ and $f\in G_2$, since $P^*=P$, we have by duality
  \[
  \|f\|_{L_p^{rc}(E)}=\|Pf\|_{L_p^{rc}(E)}\le C\sqrt{p'}\|f\|_{L_p(\Ga_{Q'})},
  \]
  where $\frac1p+\frac1{p'}=1$. By density, this inequality extends to $f\in G_p$. Note that it suffices to consider the decomposition of $f\in G_p$ in $G_p^c+G_p^r$ when we compute $\|f\|_{L_p^{rc}(E)}$. This gives the right inequality. The left inequality follows from duality and \eqref{khin0}.
\end{proof}
\begin{rem}
  In fact, the above argument also shows that $G_p$ is complemented in $L_p(\Ga_{Q'})$. Morally speaking, $G_p$ is a $\Ga_Q \dash\Ga_Q$ bimodule corresponding to differential forms of order one.
\end{rem}
\begin{comment}
  $G_p$ is complemented because $L_p(\Ga_{Q'})\subset L_p^{rc}(E)$. How about the other direction?
\end{comment}
\begin{cor}
  (a) Let $2\le p<\8$. Then for every $f\in \Dom(A)$,
  \[
  c_p^{-1}\|A^{1/2} f\|_p\le \max\{\|\Ga(f,f)^{1/2}\|_p, \|\Ga(f^*,f^*)^{1/2}\|_p\}\le K_p\|A^{1/2}f\|_p
  \]
  where $c_p=O(p^{2})$ and $K_p=O(p^{3/2})$.\\
(b) Let $1<p\le 2$. Then for every $f\in \Dom(A)$,
\[
K_{p'}^{-1} \|A^{1/2} f\|_p\le \inf_{\substack{\de(f)=g+h\\ g\in G_p^c,h\in G_p^r}} \{\|E(g^*g)^{1/2}\|_p +\|E(hh^*)^{1/2}\|_p\}\le C_p\|A^{1/2} f\|,
\]
where $K_{p'}=O(1/(p-1)^{3/2})$ and $C_p=O(1/(p-1)^2)$.
\end{cor}
\begin{proof}
Note that $\de(f)\in G_p$ if $f\in \Dom(A)$. Since $E(\de(f)^*\de(f))=\Ga(f,f)$, using Lemma \ref{khin} for $2\le p<\8$, we have
  \[
  \|\de(f)\|_p\le C\sqrt{p}\max\{\|\Ga(f,f)^{1/2}\|_p, \|\Ga(f^*,f^*)^{1/2}\|_p\}\le C\sqrt{p} \|\de(f)\|_p.
  \]
  Now apply Theorem \ref{lprie} to conclude (a). For the constants, $K_p=O(p^{3/2})$ is trivial. Since $\td K_{p'} = O(p'^{3} /(p'-1)^{3/2}) =O(p^{3/2})$, we have $c_p\le O(p^{2})$. Assertion (b) follows similarly using Lemma \ref{khin} and Theorem \ref{lprie}.
\end{proof}
Compared with Theorem \ref{lprie} proved in \cite{LP99}, this result is closer to Lust-Piquard's original formulation of the Riesz transforms on the Walsh system and the fermions given in \cite{LP98}. In particular, we get the exact order of constants as in \cite{LP98}.

\subsection{$L_p$ Poincar\'e inequalities}
It was proved by Efraim and Lust-Piquard in \cite{ELP} that the $L_p$ Poincar\'e inequalities ($2\le p<\8$)
\begin{equation}\label{poin}
  \|f-\tau_m(f)\|_p\le C\sqrt{p}\max\{\|\Ga^m(f,f)^{1/2}\|_p, \Ga^m(f^*,f^*)^{1/2}\|_p\}
\end{equation}
hold for Walsh systems and CAR algebras. In fact, the same proof also works for the general spin matrix system $\ax_m$ with some technical variants as shown in \cite{LP98}. Indeed, {Lemma 6.2--6.5} in \cite{ELP} hold for the general spin systems, from which \eqref{poin} follows. Recall that we denote by $A$ the number operator on $\Ga_Q$.

\begin{theorem}
  Let $2\le p<\8$. Then for every $f\in \Dom(A)$,
  \[
  \|f-\tau_Q(f)\|_p\le C\sqrt{p}\max\{\|\Ga(f,f)^{1/2}\|_p, \Ga(f^*,f^*)^{1/2}\|_p\}.
  \]
\end{theorem}
\begin{proof}
  Assume without loss of generality that $f=\sum_{\vec{i}} \al_{\vec{i}}w(\vec{i}) = (f^m)^\bullet$ is a finite linear combination of special Wick words. Note that $E(\de(f)^*\de(f)) =(E^m[\de^m(f^m)\de^m(f^m)])^\bullet$. Then the assertion follows from \eqref{poin} and a limiting procedure as for Theorem \ref{hypeu} and Theorem \ref{lprie} with the help of Lemma \ref{exde} and Proposition \ref{exd2}.
\end{proof}

\section{Strong solidity}
\subsection{CCAP}
Let $\Ga_q(H)$ be the $q$-Gaussian von Neumann algebra associated to a real Hilbert space $H$ with $\dim H\ge 2$; see e.g. \cite{BKS} for more information on $\Ga_q(H)$. Avsec showed that $\Ga_q(H)$ for $-1<q<1$ has the weak* completely contractive approximation property (w*CCAP) in \cite{Av}. In particular, $\Ga_q(H)$ is weakly amenable. Our goal here is to prove that $\Ga_Q$ also has w*CCAP if $\max_{1\le i,j\le N} |q_{ij}|<1$. Our argument is based on Avsec's result.

Assume $\max_{i,j} |q_{ij}|<1$. We may find $q$ such that $\max_{i,j} |q_{ij}|<q <1$. Let $Q=q \td Q$, where $\td Q=(\td q_{ij})$ satisfies $\max_{i,j}|\td q_{ij}|<1$. For $h\in H$, let $c^q(h)$ and $(c^q)^*(h)$ be the creation and annihilation operators, respectively, acting on the $q$-Fock space $\fx_q(H)$, where $\dim H = N$. We write the $q$-Gaussian variables as $s^q(h)=c^q(h)+(c^q)^*(h)$. In particular, for an orthonormal basis (o.n.b.) $(e_j)$ of $H$, we write $s_j^q=s^q(e_j)$. Similarly, we write $s^Q(h)=c^Q(h)+(c^Q)^*(h)$ for the mixed $q$-Gaussian variables of $\Ga_Q$; see \cite{LP99}. In particular, $s_j^Q=s^Q(e_j)$. We write $x_{i,j}=s^{\td Q\otimes \mathds{1}_n}(f_i\otimes e_j)$, where $(f_i)$ is an o.n.b. of $\ell_2^N$ and $(e_j)$ is an o.n.b. of $\ell_2^n$. Clearly, $x_{i,j}$'s generate $\Ga_{\td Q\otimes\mathds{1}_n}$. We first construct an ``approximate co-multiplication'' for $\Ga_Q$.
\begin{prop}\label{isom2}
  Let $\pi_\ux: \Ga_Q\to \prod_{m,\ux} \Ga_{q}(\ell_2^m)\bar\otimes \Ga_{\td Q\otimes \mathds{1}_m}$ be a $*$-homomorphism given by
  \[
  \pi_\ux(s^Q_i)=\Big( \frac1{\sqrt{m}}\sum_{k=1}^m s_k^q \otimes x_{i,k}\Big)^\bullet.
  \]
  Then $\pi_\ux$ is trace preserving. Therefore, $\Ga_Q$ is isomorphic to the von Neumann algebra generated by $\pi_\ux(s^Q_i)$.
\end{prop}
\begin{proof}
  Let $d$ be an even integer. By  the moment formula \eqref{mome},
  \begin{align*}
  &\sum_{\vec{k}\in[m]^d}\tau_q\otimes \tau_{\td Q\otimes \mathds{1}_m}[ (s^q_{k_1}\cdots s^q_{k_d})\otimes (x_{i_1,k_1}\cdots x_{i_d,k_d})]\\
  =&\sum_{\si\in P_2(d),\si\le \si(\vec{i})}\sum_{\si(\vec{k})=\si} q^{\# I(\si)}\prod_{\{r,t\}\in I(\si)}\td q(i(e_r),i(e_t))\\
  =& \sum_{\si\in P_2(d),\si\le \si(\vec{i})}\sum_{\si(\vec{k})=\si} \prod_{\{r,t\}\in I(\si)}q(i(e_r),i(e_t)),
  \end{align*}
  where $I(\si)$ is the set of inversions for the partition $\si$. Counting the number of $\vec{k}$ with $\si(\vec{k})=\si$, we have
  \[
  \tau_\ux\Big(\frac1{m^{d/2}}\sum_{\vec{k}\in[m]^d}(s^q_{k_1}\cdots s^q_{k_d}) \otimes (x_{i_1,k_1}\cdots x_{i_d,k_d}) \Big)^\bullet =\sum_{\si\in P_2(d),\si\le \si(\vec{i})}\prod_{\{r,t\}\in I(\si)}q(i(e_r),i(e_t)).
  \]
This coincides with $\tau_Q(s^Q_{i_1}\cdots s^Q_{i_d})$ given by \eqref{mome}.
\end{proof}
Now we want to understand the image of Wick words of $\Ga_Q$ under $\pi_\ux$. We need a Wick word decomposition result similar to Theorem \ref{wicdec}. For $\vec{i}\in [N]^d$, we define
\begin{equation}\label{swic2}
w^s(\vec{i})= \Big(\frac1{m^{d/2}}\sum_{\vec{k}: \si(\vec{k})\in P_1(d)} (s^q_{k_1}\cdots s^q_{k_d})\otimes (x_{i_1,k_1}\cdots x_{i_d,k_d})\Big)^\bullet.
\end{equation}
\begin{prop}\label{wicdec2}
  Following the notation of Proposition \ref{isom2}, we have
  \[
  \pi_{\ux}(s_{i_1}^Q\cdots s_{i_d}^Q) = \sum_{\substack{\si\in P_{1,2}(d)\\ \si\le \si(\vec{i})}}  w^s_\si(\vec{i}).
  \]
  Here $w^s_\si(\vec{i})= f_\si(\vec{i}) w^s(\vec{i}_{np})$,  $f_\si(\vec{i})$ and $\vec{i}_{np}$ are the same as those in Proposition \ref{inner}.
\end{prop}
\begin{proof}
  Following verbatim the argument for Theorem \ref{wicdec}, we have
  \[
  \pi_{\ux}(s_{i_1}^Q\cdots s_{i_d}^Q) =\sum_{\substack{\si\in P_{1,2}(d)}}  w^s_\si(\vec{i}).
  \]
  Here we have
  \[
  w^s_\si(\vec{i})= \Big( \frac1{m^{d/2}}\sum_{\substack{\vec{k}\in[m]^d: \si(\vec{k})= \si}} E_{\nx_{s}(\vec{k})}[(s^q_{k_1}\cdots s^q_{k_d})\otimes (x_{i_1,k_1}\cdots x_{i_d,k_d}) ]\Big)^{\bullet},
  \]
  and $\nx_{s}(\vec{k})$ is the von Neumann algebra generated by all $s^q_{k_\al}\otimes x_{i_\al,k_\al}$'s, where $k_\al$'s correspond to singleton blocks in $\vec{k}$. To simplify the conditional expectation in the ultraproduct, we denote by $\nx^1_s(\vec{k})$ and $\nx^2_s(\vec{k})$ the von Neumann algebras generated by $s_{k_\al}^q\otimes \Ga_{\td Q\otimes \mathds{1}_m}$'s and $\Ga_q(\ell_2^m)\otimes x_{i_\al,k_\al}$'s, respectively, where $k_\al$'s correspond to singleton blocks in $\vec{k}$. Clearly, $\nx_s(\vec{k})\subset \nx_s^1(\vec{k})\cap \nx_s^2(\vec{k})$. We claim that
  \begin{equation}\label{cond2}
  E_{\nx^2_s(\vec{k})}(1\otimes (x_{i_1,k_1}\cdots x_{i_d,k_d})) =
   \begin{cases}
   f_{\si,\td Q}(\vec{i}) 1\otimes (x_{j_1,l_1}\cdots x_{j_s,l_s}), \mbox{ if } \si=\si(\vec{k})\le \si(\vec{i}),\\
   0, \mbox{ otherwise,}
   \end{cases}
  \end{equation}
  where $(l_1,\cdots, l_s)$ is obtained by deleting pair blocks in $\vec{k}$, which also gives the corresponding $(j_1,\cdots, j_s)$, and
  \[
  f_{\si,\td Q}(\vec{i})=\prod_{\{r,t\}\in I_p(\si)}\td q(i(e_r),i(e_t)) \prod_{\{r,t\}\in I_{sp}(\si)}\td q(i(e_r),i(e_t)).
  \]
  Unlike the matrix models, $x_{i_\al,k_\al}$'s do not have commutation relations. We check \eqref{cond2} by calculating the inner product of $E_{\nx^2_s(\vec{k})}(1\otimes x_{i_1,k_1}\cdots x_{i_d,k_d})$ and monomials generated by $1\otimes x_{i_{\al}, k_\al}$'s in $\nx^2_s(\vec{k})$. Let $1\otimes x_{i'_1,k'_1}\cdots x_{i'_n,k'_n}\in \nx^2_s(\vec{k})$ be a monomial. Since $E_{\nx^2_s(\vec{k})}$ is trace preserving, by the moment formula \eqref{mome} for mixed $q$-Gaussian algebras,
  \begin{align*}
  &\tau_{\td Q\otimes \mathds{1}_m}[x_{i'_n,k'_n}\cdots x_{i'_1,k'_1} E_{\nx^2_s(\vec{k})}(x_{i_1,k_1}\cdots x_{i_d,k_d})] \\
  &= \begin{cases}
    f_{\si,\td Q}(\vec{i})\tau_{\td Q\otimes \mathds{1}_m} (x_{i'_n,k'_n}\cdots x_{i'_1,k'_1} x_{j_1,l_1}\cdots x_{j_s,l_s}), \quad\mbox{if  } \si=\si(\vec{k})\le \si(\vec{i}),\\
    0,\quad \mbox{otherwise.}
  \end{cases}
  \end{align*}
  Hence \eqref{cond2} is verified. Similarly, it can be checked that
  \[
  E_{\nx^1_s(\vec{k})}((s^q_{k_1}\cdots s^q_{k_d})\otimes 1)= q^{\# I_p(\si)+\# I_{sp}(\si)} s_{l_1}^q\cdots s_{l_s}^q.
  \]
  Note that $f_{\si}(\vec{i})=q^{\# I_p(\si)+\# I_{sp}(\si)}f_{\si,\td Q}(\vec{i})$. The assertion follows from the fact that $E_{\nx_s(\vec{k})}=E_{\nx_s(\vec{k})}E_{\nx^1_s(\vec{k})}E_{\nx^2_s(\vec{k})}$.
\end{proof}
\begin{prop}\label{isom3}
  $\pi_\ux$ extends to an isomorphism between $L_2(\Ga_Q)$ and $L_2\dash\Span\{w^s(\vec{i}): \vec{i}\in [N]^d, d\in \zz_+\}$.
\end{prop}
\begin{proof}
  Put $H_W= L_2\dash\Span\{w^s(\vec{i}): \vec{i}\in [N]^d, d\in \zz_+\}$. By Proposition \ref{wicdec2}, we know that $\pi_\ux(L_2(\Ga_Q))\subset H_W$. The converse containment follows from the same induction argument as for Proposition \ref{isom1}.
\end{proof}

\begin{rem}
  In fact, one can prove that $\pi_\ux(w(\vec{i}))=w^s(\vec{i})$ using the Fock space representation. Since we do not need this fact, we leave it to the reader.
\end{rem}
Now we are ready for the first main result of this section.
\begin{theorem}\label{ccap}
  $\Ga_Q$ has the weak* completely contractive approximation property for all $Q$ with $\max_{1\le i,j\le N} |q_{ij}|<1$.
\end{theorem}
\begin{proof}
  Let $H$ be a real Hilbert space and $-1<q<1$. In \cite{Av}, Avsec proved that there exists a net of finite rank maps $\varphi_\al(A)$ which converges to the identity map on $\Ga_q(H)$ in the point-weak* topology and such that $\|\varphi_\al(A)\|_{cb}\le 1+\eps$ for some prescribed $\eps$. Here $\|\cdot\|_{cb}$ is the completely bounded norm and $\varphi_\al(A)$ only depends on the number operator $A$ on $\Ga_q(H)$. Let $Q=q\td Q$ as above. Consider the following diagram:
  \[
  \xymatrix{
  \Ga_Q\ar@{^{(}->}^<<<<{\pi_\ux}[r] \ar[d]^{\psi_\al} & \prod_{m,\ux} \Ga_{q}(\ell_2^m)\bar\otimes \Ga_{\td Q\otimes \mathds{1}_m}\ar[d]^{\varphi_\al(A)\otimes \id}\\
  \Ga_Q\ar@{^{(}->}[r]^<<<<{\pi_\ux}  & \prod_{m,\ux} \Ga_{q}(\ell_2^m)\bar\otimes \Ga_{\td Q\otimes \mathds{1}_m}
  }
  \]
  where we define $\psi_\al=\pi_\ux^{-1}\circ (\varphi_\al(A)\otimes \id)\circ\pi_\ux$. Here $\varphi_\al(A)\otimes \id$ is well-defined on the ultraproduct of von Neumann algebras because it is uniformly bounded in each $\Ga_{q}(\ell_2^m)\bar\otimes \Ga_{\td Q\otimes \mathds{1}_m}$. By an argument similar to that in Section \ref{ousg}, $\varphi_\al(A)\otimes \id$ is a normal map. Note that $\psi_\al$ is well-defined because $\pi_\ux$ is injective and $\varphi_\al(A)\otimes \id$ acts as a multiplier. We claim that $\psi_\al$ is the desirable completely contractive approximation of identity. By construction, the only non-trivial thing to check is that $\psi_\al$ is finite rank. To this end, it suffices to show that $\varphi(A)_\al\otimes \id$ restricted to
  $$
  \pi_\ux(L_2(\Ga_Q))=L_2\dash\Span\{w^s(\vec{i}): \vec{i}\in[N]^d, d\in \nz \}
  $$
  is finite rank thanks to Proposition \ref{isom2} and \ref{isom3}. Since $\varphi_\al(A)$ is finite rank, suppose its range is $\Span\{s_{k_1}^q\cdots s_{k_n}^q: \si(\vec{k})\in \cup_{n\in\nz} P_1(n), \vec{k}\in B\}$ for some finite set $B$. Then the range of $\varphi_\al(A)\otimes \id |_{\pi_\ux(L_2(\Ga_Q))}$ is
  \[
  \Span\{s_{k_1}^q\cdots s_{k_n}^q\otimes x_{i_1,k_1}\cdots x_{i_n,k_n}: \si(\vec{k})\in \cup_{n\in\nz} P_1(n), \vec{k}\in B\}.
  \]
  Therefore $\varphi_\al(A)\otimes \id |_{\pi_\ux(L_2(\Ga_Q))}$ is a finite rank map.
\end{proof}

\subsection{Strong solidity}
We follow closely the argument in \cite{Av,HS11}. The strategy is to first prove a weak containment result of bimodules and then use it to prove strong solidity of $\Ga_Q$. See e.g. \cites{BO08,Av} for more details on bimodules and weak containment. For simplicity, we write $ Q'=Q\otimes \mathds{1}_2=Q\otimes {1~ 1 \choose 1~ 1}$ as in Section \ref{s:deri}. Here, we assume
\begin{equation}\label{assu5}
  \max_{1\le i,j\le N} |q_{ij}|<q^2<q<1.
\end{equation}
Recall that $L_2^0(\Ga_{Q'})$ denotes the subspace of $L_2(\Ga_{Q'})$ which consists of mean zero elements. Define the following subspaces of $L_2^0(\Ga_{Q'})$
\[
F_m=L_2\dash\Span\{w(\vec{i}):\vec{i}\in [2N]^s, s\in \nz, s\ge m, \exists i_1,\cdots, i_m \in \{N+1,\cdots, 2N\}\},
\]
and $E_m=\oplus_{k=0}^m F_k$. Clearly, $E_m^\perp$ is a $\Ga_Q\dash\Ga_Q$-subbimodule of $L_2^0(\Ga_{Q'})$. We want to show that $E_m^\perp$ is weakly contained in the coarse bimodule $L_2(\Ga_Q)\otimes L_2(\Ga_Q)$ for $m$ large enough. By Proposition \ref{isom1}, we may identify $L_2(\Ga_{Q'})$ with the Fock space $\fx_{ Q'}$. For $\xi,\eta\in L_2^0(\Ga_{Q'})$, define $\Phi_{\xi,\eta}: L_2(\Ga_Q)\to L_2(\Ga_Q)$ by
\[
\Phi_{\xi,\eta}(x)=E_{\Ga_Q}(\xi x \eta).
\]
To distinguish the left action and the right action of $\Ga_{Q'}$ on $L_2(\Ga_{ Q'})$, we write $l(h_i)$ (resp. $r(h_i)$) as the left (resp. right) creation operator associated to $h_i$ acting on the Fock space $\fx_{Q'}$, i.e.,
\[
l(h_i) (h_{j_1}\otimes\cdots\otimes h_{j_n})= h_i\otimes h_{j_1}\otimes\cdots\otimes h_{j_n},
\]
\[
r(h_i) (h_{j_1}\otimes\cdots\otimes h_{j_n})= h_{j_1}\otimes\cdots\otimes h_{j_n}\otimes h_i.
\]
Here $h_i$'s are elements in $\cz^{2N}=\cz^{N}\oplus \cz^{N}$. We write $l(h_i)^*$ (resp. $r(h_i)^*$) as the left (resp. right) annihilation operator acting on the Fock space $\fx_{Q'}$. See more details for these operations in \cites{BS94,LP99}. One can also define them following Section \ref{s:fock} after choosing an o.n.b. Denote by
$$
H'^s=\Span\{w(\vec{i})\in L_2(\Ga_{ Q'}): \vec{i}\in [2N]^s\}
\quad \text{and}\quad H^s=\Span\{w(\vec{i})\in L_2(\Ga_{Q}): \vec{i}\in [N]^s\}.
$$

\begin{lemma}\label{lrcre}
Assume \eqref{assu5}. Let $(e_{i})_{i=1}^{2N}$ be an o.n.b. of $\cz^{2N}$. Suppose $\vec{i}\in [2N]^{n_1}$ and $\vec{j}\in[2N]^{n_2}$. If there are exactly $n$ elements of $\{i_{r+1},\cdots, i_{n_1}\}$ in $\{N+1,\cdots, 2N\}$, then
\begin{align*}
  &\| E_{\Ga_Q}\big[l(e_{i_1})\cdots l(e_{i_r}) l(e_{i_{r+1}})^*\cdots
l(e_{i_{n_1}})^* r(e_{j_1})\cdots r(e_{j_s}) r(e_{j_{s+1}})^*\cdots r(e_{j_{n_2}})^*(x) ] \|_2\\
&\le C_{q,n_1,n_2} q^{(\al -(n_2-s)- (n_1-r-n))n}\|x\|_2
\end{align*}
for all $x\in H^\al$.
\end{lemma}
\begin{proof}
Note that among all possible configurations, the assertion is non-trivial only when
$$i_1,\cdots, i_r, j_{s+1}, \cdots, j_{n_2}\le N.$$
By \cite{BS94}*{Theorem 3.1},
\begin{equation}\label{anorm}
\|r(e_{j_s})^*\|\le \frac1{\sqrt{1-q}}, \quad \|l(e_{i_r})\|\le \frac{1}{\sqrt{1-q}}.
\end{equation}
We may assume without loss of generality that $r=0,s=n_2$ and estimate the norm of $l(e_{i_1})^*\cdots l(e_{i_{n_1}})^* r(e_{j_1})\cdots  r(e_{j_{n_2}}) $. The idea is that all $e_{i_r}$'s with $i_r> N$ have to pair with $e_{j_s}$'s to cancel out, and moving across the element $x$ will yield a power of $q$.  Let us assume $i_{n_1}> N$ to illustrate the argument. Note that by Remark \ref{rempr}$, H^\al$ can be identified with $(\cz^N\oplus 0)^{\otimes \al}$ via
\[
w(\vec{i})\mapsto W(e_{i_1}\otimes \cdots \otimes e_{i_\al}) \mapsto e_{i_1}\otimes \cdots \otimes e_{i_\al}.
\]
First assume $x=e_{k_1}\otimes\cdots\otimes e_{k_\al}$. Using \eqref{crann} (or the formula on p. 109 of \cite{BS94}), we find
  \begin{align*}
  l(e_{i_{n_1}})^* r(e_{j_1})\cdots r(e_{j_{n_2}}) x &= \sum_{m=1}^{n_2}\de_{i_{n_1}, j_m} \prod_{s=1}^\al q_{i_{n_1},k_{s}}\prod_{r=m+1}^{n_2} q_{i_{n_1},j_r} x \otimes e_{j_{n_2}}\otimes \cdots \otimes\check{e}_{j_m}\otimes \cdots e_{j_1}\\
  &=\sum_{m=1}^{n_2} \de_{i_{n_1}, j_m} \prod_{s=1}^\al q_{i_{n_1},k_{s}}  \prod_{r=m+1}^{n_2} q_{i_{n_1},j_r} r(e_{j_1})\cdots \check{r}(e_{j_m})\cdots r(e_{j_{n_2}}) x,
  \end{align*}
  where $\check{e}_{j_m}$ and $\check{r}(e_{j_m})$ mean that $e_{j_m}$ and ${r}(e_{j_m})$ are omitted in the expression. The difficulty is that the coefficient in front of $x$ depends on $x$. In order to extend the above equation to arbitrary $x\in H^\al$, we will find a linear operator for any fixed $m$ via deformation and enlargement of the algebra.  Define $\td q_{ij}=q_{ij}/q$ for $1\le i,j\le N$, $\td Q=(\td q_{ij})$, and
\[
P=\left(\begin{array}{cc}
Q\otimes \mathds{1}_2& \td Q\otimes \mathds{1}_2\\
\td Q\otimes \mathds{1}_2 & Q\otimes \mathds{1}_2
\end{array}
\right).
\]
Note that \eqref{assu5} implies that $\max_{ij} | p_{ij}|<q$. We can construct new von Neumann algebras $\Ga_P$ and $\Ga_{P\otimes \mathds{1}_{n_2+1}}$. Clearly, we have the following relation
\[
\Ga_Q \hookrightarrow \Ga_{Q'} \hookrightarrow  \Ga_{P} \hookrightarrow  \Ga_{P\otimes \mathds{1}_{n_2+1}}.
\]
We still denote by $E_{\Ga_Q}: \Ga_{P\otimes  \mathds{1}_{n_2+1}} \to \Ga_Q$ the conditional expectation. Let $\hat{i}_{n_1}= {i_{n_1}+2N}$,
\[
\hat{j}_r=
\begin{cases}
{j_r+2N}, \quad\text{ if } j_r>N,\\
j_r,  \quad \text{ otherwise}.
\end{cases}
\]
For fixed $m$, let
$$
\td i_{n_1}=\hat{i}_{n_1}+4mN=i_{n_1}+2N+4mN, \quad \td j_m = j_m + 4mN
$$
and $\td j_r=\hat j_r$ for $r\neq m$. In $L_2(\Ga_{P\otimes \mathds{1}_{n_2+1}})$, observing the repetition pattern in the matrix $P$, we have
\begin{align*}
&r(e_{ \td j_1})\cdots r(e_{\td j_{m-1}})  [l(e_{\td{i}_{n_1}})^* {r}(e_{\td  j_m})] r(e_{\td j_{m+1}})\cdots r(e_{ \td j_{n_2}}) x \\
&=\sum_{u=m}^{n_2} \de_{\td i_{n_1}, \td j_u} \prod_{s=1}^\al \td q_{i_{n_1},k_s}  \prod_{v=u+1}^{n_2} p_{\hat i_{n_1}, \hat j_v} r(e_{ \td j_1})\cdots r(e_{\td j_{m-1}}) \cdots \check{r}(e_{\td  j_u})\cdots r(e_{ \td j_{n_2}}) x\\
&= \prod_{s=1}^\al \td q_{i_{n_1},k_s}  \prod_{v=m+1}^{n_2} p_{\hat i_{n_1}, \hat j_v} r(e_{ \td j_1})\cdots \check{r}(e_{\td  j_m})\cdots r(e_{ \td j_{n_2}}) x,
\end{align*}
where $p_{\hat i_{n_1}, \hat j_v} =q_{i_{n_1},j_v}$ if $j_v\in \{N+1,\cdots, 2N\}$ and $p_{\hat i_{n_1}, \hat j_v} =\td q_{i_{n_1},j_v}$ otherwise. Note that the term $4mN$ is used to guarantee that $l(e_{\td i_{n_1}})^*$ only annihilates $e_{\td j_m}$. Let
$$I(m)= \{j_v: v\in\{m+1,\cdots, n_2\}, j_v \le N\}.
$$
Then
\begin{align}\label{alter}
&E_{\Ga_Q}[ l(e_{ i_{n_1}})^* r(e_{ j_1})\cdots r(e_{ j_{n_2}}) x ]  \nonumber \\
&= q^\al \sum_{m=1}^{n_2} \de_{ i_{n_1},  j_m} \prod_{s=1}^\al \td q_{i_{n_1},k_{s}}  \prod_{r=m+1}^{n_2}  q_{ i_{n_1}, j_r}E_{\Ga_Q}[ r(e_{ j_1})\cdots \check{r}(e_{ j_m})\cdots r(e_{ j_{n_2}}) x] \nonumber \\
&= q^\al \sum_{m=1}^{n_2} \de_{i_{n_1},j_m} q^{\# I(m)} \prod_{s=1}^\al \td q_{i_{n_1},k_s}  \prod_{v=m+1}^{n_2} p_{\hat i_{n_1}, \hat j_v} E_{\Ga_Q} [ r(e_{ \td j_1})\cdots \check{r}(e_{\td  j_m})\cdots r(e_{ \td j_{n_2}}) x ]\nonumber \\
&=q^\al \sum_{m=1}^{n_2} \de_{i_{n_1},j_m} q^{\# I(m)}E_{\Ga_Q} [ r(e_{ \td j_1})\cdots r(e_{\td j_{m-1}})  [l(e_{\td{i}_{n_1}})^* {r}(e_{\td  j_m})] r(e_{\td j_{m+1}})\cdots r(e_{ \td j_{n_2}}) x ] .
\end{align}
Here the conditional expectation is used in the second equality so that the change in $\vec{i}$ and $\vec{j}$ will not affect the resultant value in $L_2(\Ga_Q)$. Note that the summand in \eqref{alter} does not depend on $x$ for each fixed $m$. By linearity, \eqref{alter} holds for any $x\in H^\al$. We deduce from \eqref{anorm} and the  triangle inequality that
\[
\| E_{\Ga_Q}[ l(e_{ i_{n_1}})^* r(e_{ j_1})\cdots r(e_{ j_{n_2}}) x ]  \|_2\le C_{q,n_2} q^\al\|x\|_2
\]
for all $x\in H^\al$. Since $l(e_{ i_{n_1}})^* r(e_{ j_1})\cdots r(e_{ j_{n_2}}) x$ is a linear combination of words with fixed length, the above argument can be easily extended to handle more than one annihilators. To get a norm estimate on $E_{\Ga_Q}[l(e_{i_1})^*\cdots l(e_{i_{n_1}})^* r(e_{j_1})\cdots  r(e_{j_{n_2}}) (x)]$, it suffices to consider the configuration which yields the minimal power of $q$. This case occurs when
\[
i_1,\cdots, i_n, j_{n_2-n+1}, \cdots, j_{n_2} \in \{N+1,\cdots, 2N\}.
\]
In this situation, $l(e_{i_1})^*,\cdots, l(e_{i_n})^*$ need to cross at least $\al-(n_1-n)$ terms to cancel with $e_{j_s}$'s. This gives $q^{[\al-(n_1-n)]n}$. Using \eqref{anorm} to estimate the norm of $l(e_{i_{n+1}})^*\cdots l(e_{i_{n_1}})^*$ yields a constant $C_{q,n_1}$. We proceed in this way to complete the proof.
\end{proof}
We will use the normal form theorem of Wick products \cites{BKS, Kr00} to estimate the norm of $\Phi_{\xi,\eta}$. This is achieved via the following result.
\begin{lemma}\label{schest}
  Assume \eqref{assu5}. Let $\xi\in H'^{n_1}\cap F_n$ and $\eta\in H'^{n_2}\cap F_n$. Then for $\al>2(n_1+n_2)$ and $x\in H^\al$, we have
  \[
  \|\Phi_{\xi,\eta}(x) \| _2 \le C_{q,\xi,\eta}q^{n\al/2}\|x\|_2.
  \]
  Moreover, $\Phi_{\xi,\eta}(x)\in \oplus_{\bt=\al-n_1-n_2+2n}^{\al+n_1+n_2-2n} H^\bt$. \
 \end{lemma}

\begin{proof}
  First we assume $\xi=w(\vec{i}), \eta=w(\vec{j})$ and identify $x$ as a vector in $(\cz^N\oplus 0)^\al$. By the normal form theorem of Wick products \cite{BKS} and \cite{Kr00}*{Theorem 1}, we have
  \begin{align*}
    w(\vec{i})&=W(e_{i_1}\otimes \cdots \otimes e_{i_{n_1}}) \\
    &=\sum_{r=0}^{n_1}\sum_{\si\in S_{n_1}/(S_r\times S_{n_1-r})} K(Q,\si) l(e_{\si(i_1)})\cdots l(e_{\si(i_r)}) l(e_{\si(i_{r+1})})^*\cdots l(e_{\si(i_{n_1})})^*,
  \end{align*}
where $\si(i_r)=i_{\si^{-1}(r)}$, and $K(Q,\si)$ is a product of certain entries of $Q$ and only depending on $Q$ and $\si$. The precise value of $K(Q,\si)$ is irrelevant here. We only need the fact that $|K(Q,\si)|\le C_{q, n_1}$ for some constant $C_{q,n_1}$ depending on $q$ and $n_1$. We have a similar formula for $w(\vec{j})$. It follows that
\begin{align}
\Phi_{\xi,\eta}(x) = &\sum_{r=0}^{n_1}\sum_{s=0}^{n_2}\sum_{\substack{\si\in S_{n_1}/(S_r\times S_{n_1-r})\\ \pi\in S_{n_2}/(S_s\times S_{n_2-s})}} K(Q,\si) K(Q,\pi) E_{\Ga_Q}\big[l(e_{\si(i_1)})\cdots l(e_{\si(i_r)}) l(e_{\si(i_{r+1})})^*\cdots \nonumber \\
& l(e_{\si(i_{n_1})})^* r(e_{\pi(j_1)})\cdots r(e_{\pi(j_s)}) r(e_{\pi(j_{s+1})})^*\cdots r(e_{\pi(j_{n_2})})^*(x)\big].\label{xieta}
\end{align}
By Lemma \ref{lrcre},
\begin{align*}
\| E_{\Ga_Q}&\big[ l(e_{\si(i_1)})\cdots l(e_{\si(i_r)}) l(e_{\si(i_{r+1})})^*\cdots  l(e_{\si(i_{n_1})})^*r(e_{\pi(j_1)})\cdots r(e_{\pi(j_s)})\\
& r(e_{\pi(j_{s+1})})^*\cdots r(e_{\pi(j_{n_2})})^*(x)\big]\|_2 \le  C_{q, n_1,n_2} q^{(\al -(n_2-s)- (n_1-r-n))n} \|x\|_2.
\end{align*}
Since $\al-n_1-n_2+s+r+n\ge \frac12 \al $, it follows from the triangle inequality that
\[
\|\Phi_{\xi,\eta}(x) \|_2 \le C_{q,n_1,n_2} q^{n\al/2}\|x\|_2.
\]
Now suppose $\xi,\eta$ are linear combinations of special Wick words. Using the triangle inequality again, we have proved the first assertion. As for the range of $\Phi_{\xi,\eta}$, a moment of thought shows that the summand in \eqref{xieta} has length $\al-n_1-n_2+2s+2r$ and that $0\le r\le n_1- n, n\le s\le n_2$ because we must have $\si(i_1),\cdots, \si(i_r), \pi(j_{s+1}), \cdots, \pi(j_{n_2})\le N$ so that the right-hand side of \eqref{xieta} is nonzero. This gives the ``moreover'' part of the lemma.
\end{proof}

%\begin{defn}
%Let $K=\oplus_{n=0}^{\infty} K_n$ be a decomposition of Hilbert space. An operator $T:K\to K$ is said to have finite propagation if there exists a $L$, and for every $n\in \nz$ there exists a subset $A(n)\subset \nz$ of cardinality $L$ such that
% \[  T(K_n)\subset \oplus_{j\in A(n)} K_j.\]
%\end{defn}

\begin{lemma}\label{HS}
Let $K=\oplus_{n=0}^{\infty} K_n$  and $T:K\to K$ be an operator such that
 \begin{enumerate}
 \item[i)] $\dim(K_n)\le d^n$; {\rm ii)} $\|T|_{K_n}\|\le C \al^n$ for $n\ge n_0$; {\rm iii)} $\al^2d<1$.
 \end{enumerate}
Then $T$ is Hilbert--Schmidt.
\end{lemma}

\begin{proof}
 Let $P_n:K\to K_n$ be the orthogonal projection. Then
 \begin{align*}
 \tr(T^*T) &=  \sum_n \tr((TP_n)^*TP_n)
 \le \sum_n \|TP_n\|^2 d^n
 \le C \sum_n \al^{2n}d^n
 \end{align*}
Since the series is absolutely convergent the assertion follows immediately.
\end{proof}

\begin{lemma}\label{schat}
 Let $\xi,\eta\in F_n$ and $n> -\ln N/\ln q$. Then $\Phi_{\xi,\eta}: L_2(\Ga_Q)\to L_2(\Ga_Q)$ is Hilbert--Schmidt.
\end{lemma}
\begin{proof}
  Write $L_2(\Ga_Q)=\oplus_{s=0}^\8 H^s$. Then $\dim(H^\al)\le N^\al$ and $q^n N<1$. By Lemma \ref{schest}, we have
  $$\|\Phi_{\xi,\eta}|_{H^\al}\|\le C_{q,\xi,\eta} (q^{n/2})^\al.
  $$
The assertion follows from Lemma \ref{HS}.
\end{proof}

\begin{prop}\label{wkcon}
 Let $n> -\frac{\ln N}{\ln q}$. Then $E_{n-1}^\perp$ is weakly contained in the coarse bimodule $L_2(\Ga_Q)\otimes L_2(\Ga_Q)$.
\end{prop}
\begin{proof}
The proof is given in \cite{Av}*{Proposition 4.1} using Lemma \ref{schat}.
\end{proof}

Let $R_t: \rz^{N}\oplus \rz^N \to \rz^{N}\oplus \rz^N$ be the orthogonal transform
\[
R_t=\Big( \begin{array}{cc}
e^{-t}\id& -\sqrt{1-e^{-2t}}\id\\
\sqrt{1-e^{-2t}} \id & e^{-t} \id
\end{array}\Big),
\]
where $\id: \rz^N\to \rz^N$ is the identity operator and we understand the canonical o.n.b. in $0\oplus \rz^N$ is $\{e_{N+1},\cdots, e_{2N}\}$. Recall from \cite{LP99}*{Lemma 3.1} that there is a second quantization functor $\Ga_Q$ which sends the category of Hilbert spaces to the category of mixed $q$-Gaussian algebras. Let $\al_t = \Ga_Q(R_t)$. Then $\al_t$ is a trace preserving $*$-automorphism on $\Ga_{ Q'}$ and extends to an isometry on $L_2(\Ga_{Q'})$. It is easy to check that $T_t=E_{\Ga_Q}\circ\al_t$ coincides with the Ornstein--Uhlenbeck semigroup on $\Ga_Q$ defined in Section \ref{ousg}. The following is a modification of Popa's s-malleable deformation estimate \cite{Po08}*{Lemma 2.1}. The proof modifies slightly that of \cite{Av}*{Proposition 5.1}. We provide the difference here for the reader's convenience. Recall that $L_2^0(\Ga_{Q'})=\oplus_{m=1}^\8 {H'}^m$.
\begin{prop}\label{meana}
  Let  $P_k:  L_2^0(\Ga_{Q'}) \to E_k^\perp$ be the orthogonal projection. Then for $k\ge 1$, we have
  \[
  \|(\al_{t^k}-\id)(x)\|_2\le C_k\|P_{k-1}\al_t(x)\|_2
  \]
  for $x\in \oplus_{n=k}^\8 H^n \subset L_2(\Ga_Q)$ and $t<2^{-k}$.
\end{prop}
\begin{proof}
Note that $\al_{t^k}-\id$ and $P_{k-1}\al_t$ preserve the length of $n$-tensors for $n\ge k$ and $t>0$. It suffices to prove the  assertion for $x\in H^n$ with $n\ge k$. Identify $H^n$ with $(\cz^N\oplus 0)^{\otimes n}$. Let $x=e_{i_1}\otimes \cdots\otimes e_{i_n}$ and $y=e_{j_1}\otimes \cdots \otimes e_{j_n}$. Then
\[
\lge P_{k-1} \al_t(x), P_{k-1}\al_t(y)\rge =\sum_{m=k}^n \lge P_{F_m} \al_t(x), P_{F_m} \al_t(y)\rge
\]
where the inner product is given by Proposition \ref{swinn} and $P_{F_m}: L_2^0(\Ga_{Q'})\to F_m$ is the orthogonal projection. By the second quantization \cite{LP99}*{Lemma 3.1}, we know that
\[
\al_t(e_{i_1}\otimes \cdots\otimes e_{i_n})= (e^{-t} e_{i_1}+\sqrt{1-e^{-2t}} e_{N+i_1}) \otimes \cdots \otimes (e^{-t} e_{i_n} +\sqrt{1-e^{-2t}} e_{N+i_n}).
\]
It follows that
\[
P_{F_m}\al_t(x) = \sum_{B\subset \{1,\cdots, n\}, |B|=m} (1-e^{-2t})^{m/2} e^{-t(n-m)} e_{\pi_B(i_1)}\otimes \cdots \otimes e_{\pi_B(i_n)},
\]
where $\pi_B (i_k)=N+i_k$ for $k\in B$ and $\pi_B(i_k)= i_k$ otherwise. Similarly, we get
\[
P_{F_m}\al_t(y)=\sum_{C\subset \{1,\cdots, n\}, |C|=m} (1-e^{-2t})^{m/2} e^{-t(n-m)} e_{\pi_C(j_1)}\otimes \cdots \otimes e_{\pi_C(j_n)},
\]
where $\pi_C (j_k)=N+j_k$ for $k\in C$ and $\pi_C(j_k)= j_k$ otherwise. By Proposition \ref{swinn}, $\lge P_{F_m} \al_t(x), P_{F_m} \al_t(y)\rge$ is nonzero only if $\{\pi_B(i_1),\cdots, \pi_B(i_n)\}$ and $\{\pi_C(j_1),\cdots, \pi_C(j_n)\}$ are equal as multisets. Hence, the indices in $B$ have to be paired with the indices in $C$ when we compute $\lge e_{\pi_B(i_1)}\otimes \cdots \otimes e_{\pi_B(i_n)}, e_{\pi_C(j_1)}\otimes \cdots \otimes e_{\pi_C(j_n)}\rge$ using Proposition \ref{swinn}. For every fixed $B$, pairing all the possible $C$ with $B$ and the corresponding $C^c$ with $B^c$ gives all the bipartite partitions of $\vec{i}\sqcup \vec{j}$. Using Proposition \ref{swinn} again, we see that
\[
\lge P_{F_m}\al_t(x), P_{F_m} \al_t(y)\rge = (1-e^{-2t})^{m} e^{-2t(n-m)}\sum_{B\subset \{1,\cdots, n\}, |B|=m} \lge x,y\rge.
\]
By linearity, this identity holds for arbitrary $x,y\in H^n$. Hence,
\[
\lge P_{k-1} \al_t(x), P_{k-1}\al_t(y)\rge = \sum_{m=k}^n (1-e^{-2t})^{m} e^{-2t(n-m)} {n\choose m}\lge x,y\rge.
\]
Since the Ornstein--Uhlenbeck semigroup $T_t$ is self-adjoint on $L_2(\Ga_Q)$ and $\al_t$ is trace preserving, we have for $x,y\in (\cz^N\oplus 0)^{\otimes n}$,
\[
\lge (\al_{t^k}-\id)(x), (\al_{t^k}-\id)(y)\rge = 2(\lge x,y\rge -\lge x, T_{t^k}(y)\rge ) = 2(1-e^{-nt^k})\lge x,y\rge.
\]
The rest of the proof is just numerical estimate, which is provided in the proof of \cite{Av}*{Proposition 5.1}.
\end{proof}
The following is the main result of this section.
\begin{theorem}
  Let $Q$ be a real symmetric $N\times N$ matrix with $\max_{1\le i,j\le N}|q_{ij}|<1$ and $N<\8$. Then $\Ga_Q$ is strongly solid.
\end{theorem}
\begin{proof}
  The proof is the same as that of \cite{Av}*{Theorem B}, with the help of Theorem \ref{ccap}, Proposition \ref{wkcon} and Proposition \ref{meana}. The argument in \cite{Av} follows literally the same strategy as that of \cite{HS11}*{Theorem 3.5}, which in turn is a suitable modification of \cites{OP10,OP10b}.
\end{proof}

\begin{appendix}
\section{Speicher's central limit theorem}
\begin{proof}[Proof of Theorem \ref{clt}]
 The proof is rephrased from \cite{Sp93} and also follows \cite{JPP}. We first show that the convergence holds on average, and then prove almost sure convergence using the Borel--Cantelli lemma. We write
 \begin{align}\label{trac}
   &\tau_m(\td x_{i_1}(m)\cdots \td x_{i_s}(m)) = \frac1{m^{s/2}}\sum_{\vec{k}\in [m]^s}\tau_m(x_{i_1}(k_1)\cdots x_{i_s}(k_s))\\
   & = \frac1{m^{s/2}}\sum_{\si\in P(s)}\sum_{\substack{\vec{k}\in[m]^s\\ \si(\vec{k})=\si}} \tau_m(x_{i_1}(k_1)\cdots x_{i_s}(k_s))=:\frac1{m^{s/2}}\sum_{\si\in P(s)} \De_\si\nonumber.
 \end{align}
By the commutation/anticommutation relation, $\De_\si=0$ if $\si$ contains a singleton. Note that $|\tau_m(x_{i_1}(k_1)\cdots x_{i_s}(k_s))|\le 1$. If $\si$ has $r$ blocks, then $\De_\si\le m(m-1)\cdots (m-r+1)$. Hence,
\begin{equation}\label{nonp}
  \lim_{m\to\8}\frac1{m^{s/2}}\De_\si = 0
\end{equation}
for $r<s/2$ and thus for $\si\in P(s)\setminus P_2(s)$ since the singleton case is automatically true. Our argument so far is independent of $\om\in \Om$ so that \eqref{nonp} holds for all $\om\in\Om$. The theorem follows immediately from \eqref{nonp} if $s$ is odd. Therefore, we only need to consider $\si\in P_2(s)$ in \eqref{trac}. To this end, we write $\si=\{\{e_1,z_1\}, \cdots, \{e_{s/2}, z_{s/2}\}\}$. Since $\si(\vec{k})=\si$ is a pair partition, if $k_j=k_l$, then $i_j=i_l$ in order for $x_{i_j}(k_j)$ and $x_{i_l}(k_l)$ to cancel out. Hence we may assume $\si\le \si(\vec{i})$. In this case, if $\{r,t\}\in I(\si)$, then we have to switch $x_{i({e_r})}(k({e_r}))$ and $x_{i({e_t})}(k({e_t}))$ to cancel the corresponding $x_{i({z_r})}(k({z_r}))$ and $x_{i({z_t})}(k({z_t}))$ terms, which yields
\begin{equation}\label{sigpro}
\tau_m(x_{i_1}(k_1)\cdots x_{i_s}(k_s)) = \prod_{\{r,t\}\in I(\si)} \eps([i(e_r),k(e_r)], [i(e_t),k(e_t)]).
\end{equation}
By independence and counting the elements in $\{\vec{k}\in[m]^s| \si(\vec{k})=\si\}$, we find
\begin{align*}
  &\ez(\De_\si)=\sum_{\substack{\vec{k}\in[m]^s\\ \si(\vec{k})=\si}} \prod_{\{r,t\}\in I(\si)} q(i(e_r),i(e_t)) \\
 &= m(m-1)\cdots (m-s/2+1) \prod_{\{r,t\}\in I(\si)} q(i(e_r),i(e_t)).
\end{align*}
Combining together, we have
\begin{equation}\label{etmom}
  \lim_{m\to\8}\ez(\tau_m(\td x_{i_1}(m)\cdots \td x_{i_s}(m)))= \sum_{\substack{\si\in P_2(s) \\ \si\le \si(\vec{i})}} \prod_{\{r,t\}\in I(\si)} q(i(e_r),i(e_t)).
\end{equation}
It remains to prove the almost sure convergence. Put $X_m=\tau_m(\td x_{i_1}(m)\cdots \td x_{i_s}(m)))$ and $E_m(\al)=\{\om: |X_m-\ez X_m|\ge \al\}$. Then we only need to show $\pz(\limsup_m E_m(\al)) = 0$. By the Borel--Cantelli lemma and Chebyshev's inequality, it suffices to show that for any $\al>0$,
$$\sum_{m=1}^\8\pz(E_m(\al))\le \frac1{\al^2}\sum_{m=1}^\8\Var (X_m)<\8,
$$
where $\Var(X_m)$ is the variance of $X_m$. Decompose $X_m=Y_m+Z_m$ where $Y_m$ corresponds to sum over all pair partitions in \eqref{trac} and $Z_m=X_m-Y_m$. Since \eqref{nonp} holds for $\si\in P(s)\setminus P_2(s)$, we have $\lim_{m\to \8} X_m- Y_m=0$ for all $\om\in \Om$. But $Z_m$ is uniformly bounded, $\lim_{m\to \8}\Var(Z_m)=0$. Therefore, we only need to show that $\sum_{m=1}^\8 \Var(Y_m)<\8$. Write
\begin{align*}
  \Var(Y_m) = \frac1{m^s}\sum_{\si,\pi\in P_2(s)}\sum_{\vec{k}: \si(\vec{k})=\si\atop \vec{l}: \si(\vec{l})=\pi} V_{\vec{k},\vec{l}},
\end{align*}
where
\begin{align}
  V_{\vec{k},\vec{l}}&= \ez [ \tau_m(x_{i_1}(k_1)\cdots x_{i_s}(k_s))\tau_m(x_{i_1}(l_1)\cdots x_{i_s}(l_s))]\nonumber\\
  &\quad - \ez[\tau_m(x_{i_1}(k_1)\cdots x_{i_s}(k_s))]\ez[\tau_m(x_{i_1}(l_1)\cdots x_{i_s}(l_s))]\nonumber\\
  &= \ez\Big(\prod_{\{r,t\}\in I(\si)} \eps([i(e_r),k(e_r)], [i(e_t),k(e_t)]) \prod_{\{r',t'\}\in I(\pi)}\eps([i(e_{r'}),l(e_{r'})], [i(e_{t'}),l(e_{t'})])\Big)\label{expe}\\
  &\quad -\prod_{\{r,t\}\in I(\si)} q(i(e_r),i(e_t))\prod_{\{r',t'\}\in I(\pi)}q(i(e_{r'}),i(e_{t'})).\nonumber
\end{align}
Let us analyze the product in \eqref{expe}. If $\{k(e_r),k(e_t)\}\neq \{l(e_{r'}),l(e_{t'})\}$ for all $\{r,t\}\in I(\si)$ and $\{r',t'\}\in I(\pi)$, then $V_{\vec{k},\vec{l}}=0$. In order to contribute for $\Var(Y_m)$, there exists at least one pair $\{r,t\}\in I(\si)$ and one pair $\{r',t'\}\in I(\pi)$ such that $\{k(e_r),k(e_t)\} =\{l(e_{r'}),l(e_{t'})\}$. In this case, we have
\[
\#\{\vec{k},\vec{l}: \si(\vec{k})=\si, \si(\vec{l})=\pi\}\le m^{s/2}m^{s/2-2}=m^{s-2}.
\]
Note that $|V_{\vec{k},\vec{l}}|\le 1$ and $C(s):=[\#P_2(s)]^2$ does not depend on $m$. It follows that
\[
 \sum_{m=1}^\8 \Var(Y_m)\le \sum_{m=1}^\8 \frac{C(s)}{m^2}<\8,
 \]
 as desired.
\end{proof}
\begin{rem}\label{cltrk}
In the above argument, we assumed that $\eps((i,k),(j,l))$ are independent for different indices. However, the independence assumption can be weakened if the structure matrix is of the form $Q\otimes\mathds{1}_n$, where $Q$ is an $N\times N$ symmetric matrix with entries in $[-1,1]$. In this case we require that $\eps((i,k),(j,l))$'s are independent (up to symmetric assumption) with \eqref{qprob} for $(i,k),(j,l)\in[N]\times \nz$ and then
\begin{equation}\label{qrepi}
\eps((i+\al N, k), (j+\bt N, l)) = \eps((i,k),(j,l))
\end{equation}
for $\al,\bt=1,\cdots, n-1$. In other words, $\eps=\eps|_{1\le i, j\le N} \otimes \mathds{1}_n$. To verify the claim, we only need to show the dependence introduced in \eqref{qrepi} will not destroy the proof of Theorem \ref{clt}. Indeed, by \eqref{nonp} it suffices to consider pair partitions. Suppose $i_\bt = i_\al+N$. Then $\al$ and  $\bt$ are not in the same pair block of $\si(\vec{i})$. It follows that $k_\al\neq k_\bt$ since $\si(\vec{k})\le \si(\vec{i})$. (If $k_\al=k_\bt$, then $i_\al=i_\bt$ in order for $x_{i_\al}(k_\al)$ and $x_{i_\bt}(k_\bt)$ to cancel.) Hence, the random signs in \eqref{sigpro} are pairwise different. Note that unlike the case in the proof of Theorem \ref{clt}, now we may have
\[
\eps((i_\al, k_\al), (i_\ga, k_\ga))\quad \text{and} \quad \eps((i_\bt,k_\bt),(i_\ga,k_\ga)) = \eps((i_\al, k_\bt),(i_\ga,k_\ga))
\]
in \eqref{sigpro}, but the two random signs are not equal because $k_\al\neq k_\bt$. In other words, the second coordinates $(k_\al, k_\ga)$ in  $\eps((i_\al, k_\al), (i_\ga, k_\ga))$ are never the same for random signs in \eqref{sigpro} even under the weaker condition \eqref{qrepi} so that $(i_\al, i_\ga)$ may be the same for different random signs. The point is that the independence structure in the proof of Theorem \ref{clt} is given via the second coordinates $k_\al$'s. The rest of the argument is the same as for Theorem \ref{clt}. We invite the interested reader to consider the simplest case $Q=q \mathds{1}_N$. In this case we can take $\eps((i,k),(j,l))= \eps((i,k),(i,l))$ and require $\eps((i,k),(i,l))$ to be independent for different $k$ and $l$ up to symmetry.

By this remark, the moment formula \eqref{mome} remains valid with the weaker condition \eqref{qrepi}. The same discussion applies in other parts of the paper when the CLT argument is invoked with \eqref{qrepi}. This subtlety is crucial for our limiting argument in Section \ref{anaprop}.
\end{rem}
\end{appendix}

\section*{Acknowledgement}
M.J. was partially supported by NSF Grant DMS-1201886 and DMS-1501103. A large part of this work was done when Q.Z. was a graduate student in University of Illinois. He thanks the financial support from the graduate college dissertation fellowship. He also thanks the financial support from the Center of Mathematical Sciences and Applications at Harvard University.

\bibliographystyle{alpha}
\bibliography{mixed}
\end{document}